\documentclass[oneside,a4paper,11pt,notitlepage]{article}
\usepackage[left=0.8in, right=0.8in, top=1.5in, bottom=1.5in]{geometry}

\usepackage[T1]{fontenc} % codifica dei font
\usepackage[utf8]{inputenc} % lettere accentate da tastiera
\usepackage[english]{babel}
\usepackage{amssymb}
\usepackage{lipsum} % genera testo fittizio
\usepackage{lmodern}
\usepackage{amsmath}
\usepackage{amsthm}
\usepackage{bm}
\usepackage{mathtools}
\usepackage{braket}
\usepackage{esint}
\newcommand{\abs}[1]{{\left|#1\right|}}
\newcommand{\norma}[1]{{\left\Vert#1\right\Vert}}

\usepackage{enumitem}
\usepackage{booktabs}
\usepackage{graphicx}
\usepackage{tikz}
\usetikzlibrary{patterns}
\usepackage{multicol}
\usepackage{caption}
\usepackage[skins,theorems]{tcolorbox}
\tcbset{highlight math style={enhanced,
colframe=black,colback=white,arc=0pt,boxrule=1pt}}
\def\Xint#1{\mathchoice
    {\XXint\displaystyle\textstyle{#1}}%
    {\XXint\textstyle\scriptstyle{#1}}%https://it.overleaf.com/project/648322e9deb9b0e833678999
    {\XXint\scriptstyle\scriptscriptstyle{#1}}%
    {\XXint\scriptscriptstyle\scriptscriptstyle{#1}}%
      \!\int}
\def\XXint#1#2#3{{\setbox0=\hbox{$#1{#2#3}{\int}$}
    \vcenter{\hbox{$#2#3$}}\kern-.5\wd0}}
\def\dashint{\Xint-}
\theoremstyle{definition}
\newtheorem{definizione}{Definition}[section]
\theoremstyle{plain}
\newtheorem{teorema}{Theorem}[section]

\newtheorem{lemma}[teorema]{Lemma}
\newtheorem{prop}[teorema]{Proposition}

\theoremstyle{definition}
\newtheorem{esempio}{Example}[section]
\newtheorem{oss}[esempio]{Remark}

\newtheorem*{open*}{Open problems}
\DeclareMathOperator{\R}{\mathbb{R}}

\makeatletter
\newcommand{\myfootnote}[2]{\begingroup
	\def\@makefnmark{}%
	\addtocounter{footnote}{-1}%
	\footnote{\textbf{#1} #2}
	\endgroup}
\makeatother

\usepackage{hyperref}
\hypersetup{linktoc=none, bookmarksnumbered, colorlinks=true, linkcolor=magenta, citecolor=cyan}

\title{The Talenti comparison result in a quantitative form}
\author{Vincenzo Amato, Rosa Barbato, Alba Lia Masiello, Gloria Paoli}
\date{}

\begin{document}

\maketitle
 \begin{abstract}
%Comparison results of Talenti type for Elliptic Problems with Dirichlet boundary conditions have been widely investigated in the last decades, also for other boundary conditions.

In this paper, we obtain a quantitative version of the classical comparison result of Talenti for elliptic problems with Dirichlet boundary conditions. The key role is played by quantitative versions of the  P\'olya-Szeg\H o inequality and of the    Hardy-Littlewood inequality. 
\end{abstract}

%\Addresses 
\section{Introduction}

Symmetrization techniques in the context of qualitative properties of solutions to second-order elliptic boundary value problems have been widely studied in the last decades.  In  the seminal paper \cite{talenti76}, Talenti considers an open and bounded set $\Omega\subset\mathbb{R}^n$ and the problem

\begin{equation}
\label{mainproblem}
    \begin{cases}
        -\Delta u =f &\text{in}\;\Omega\\
        u=0 &\text{on}\;\partial\Omega.
    \end{cases}
\end{equation}
where $f\in L^{\frac{2n}{n+2}}(\Omega)$ if $n>2$, $f\in L^p(\Omega)$, $p>1$ if $n=2$, and he  proves that it is possible to compare the solution to \eqref{mainproblem} with the solution to the symmetrized problem
\begin{equation}
\label{symmetricproblem}
    \begin{cases}
        -\Delta v =f^{\sharp} &\text{in}\;\Omega^{\sharp}\\
        v=0 &\text{on}\;\partial\Omega^{\sharp},
    \end{cases}
\end{equation}
where $\Omega^{\sharp}$ is the ball centered in the origin with the same measure as $\Omega$ and $f^\sharp$ is the Schwarz rearrangement of $f$ (see Definition \ref{rear}). More precisely,  the following point-wise comparison  is proved
%it was proved the best comparison one can obtain, the point-wise comparison 
\begin{equation}\label{tal}
    u^\sharp(x)\le v(x) \quad  \forall x \in \Omega^\sharp, 
\end{equation}
and it also holds for a generic uniformly elliptic linear operator in divergence form.  

The result by Talenti is a very powerful tool, as it adapts well to solving various types of variational problems. For instance, let us mention Saint Venant's conjecture, which asserts that the torsional rigidity, that is the $L^1-$norm of the solution to \eqref{mainproblem} in the case $f\equiv 1$, is maximum on balls among sets of given measure. This conjecture can be solved  by integrating \eqref{tal}. 

Over the years, the topic of comparison results has gained more and more interest. A version of the aforementioned result of Talenti for nonlinear operators in divergence form is contained in \cite{T2}, which includes the case of the $p-$Laplace operator as a special instance; see also \cite{betta_mercaldo}. Further extensions can be found, for instance,  in \cite{AFLT} for anisotropic elliptic operators, in \cite{ALT} for the parabolic case,  in \cite{AB, T3} for higher-order operators,  and in \cite{diaz},  where the Steiner symmetrization is used in place of the Schwarz one.
%As far as the Robin boundary conditions and comparison results, we refer to 
%As \eqref{tal} is in force, one can solve the Saint-Venant conjecture by simply integrating.  
%Even if for a long time it was believed that comparison results could not be proved by means of spherical rearrangement argument when dealing with Robin boundary conditions, i
In the recent paper \cite{ANT}, the 
authors consider problem \eqref{mainproblem} replacing the Dirichlet boundary conditions with the Robin one
and they prove a comparison result involving Lorentz norms of $u$ and $v$ whenever $f$ is a non-negative function in $L^2(\Omega)$. In particular, in the case $f\equiv 1$ and $n=2$, they recover the point-wise inequality as in \cite{talenti76}.
Generalizations of the results contained in \cite{ANT} can be found for the anisotropic case in \cite{San2, amato2022isoperimetric}, for mixed boundary conditions in \cite{ACNT}, in the case of the Hermite operator in \cite{nunzia2022sharp} and in the nonlinear case \cite{AGM}. 
For further generalizations and  general overviews of comparison results we refer, for instance, to \cite{alvino_onsome, chiti, cianchi_fusco, mazja_russo, kesavan_articolo}
 and to the references therein, not claiming to give here a comprehensive list of all the results developed in this context. 
% We would like to stress that this list of references does not claim to be exhaustive. 
%We are aware that it is impossible to make a comprehensive list of all the results developed in this context.

Another interesting question about the topic is trying to understand if comparison inequalities are rigid, as it has been done in \cite{lions_remark}. 
Indeed, the authors characterize the equality in \eqref{tal} in the case $f\ge 0$, proving that if equality occurs, then $\Omega$ is a ball, $u$ is radially symmetric and decreasing, and $f=f^\sharp$.
Rigidity results for a generic linear, elliptic second-order operator can be found in \cite{posteraro} and \cite{kesavan} and for Robin boundary conditions in \cite{NOI,masiello2023rigidity}.% citare i due nostri  

%\textcolor{magenta}{Mettere qualcosa sulle quantitative in generale.}
Once a comparison result and a rigidity result hold, it is natural to ask whether the comparison result can be improved in a quantitative version. More precisely, 
considering the rigidity result in \cite{lions_remark}, 
 if equality \emph{almost} holds in \eqref{tal}, is it true that the set $\Omega$ is \emph{almost} a ball, and the function $u$ and $f$ are \emph{almost} radially symmetric and decreasing? 
The main objective of this paper is to study the stability of \eqref{tal} and to answer the following question: if $u^\sharp$ is  \textit{close} to $v$ in some sense, can we say that the set $\Omega$ and the functions $u$ and $f$ are  \textit{almost} symmetric? One of the difficulties we have to deal with is to give the most suitable notions of "closeness" for this type of problem. 
%understand what is the  right de(most suitable) 

More precisely, we state our main result in the following way.
\begin{teorema}
\label{main_theorem} 
Let $\Omega$ be a bounded open set of $\R^n$ and let $f$ be a non-negative function in $ L^2(\Omega)$.  Suppose that $u$ and $v$ are the solutions to \eqref{mainproblem} and \eqref{symmetricproblem}, respectively. Then, there exist some positive constants $ \theta_1=\theta_1(n), \theta_2=\theta_2(n)$ and    $ C_1:= C_1(n, \abs{\Omega}, f^\sharp),\, C_2:=  C_2(n, \abs{\Omega}, f^\sharp),\, C_3:=  C_3(n, \abs{\Omega}, f^\sharp),$ such that
\begin{equation}\label{esti_tot}
     C_1 \alpha^3(\Omega)+ C_2 \inf_{x_0\in\R^n}\norma{u-u^\sharp(\cdot+x_0)}_{L^1(\R^n)}^{ \theta_1} +C_3\inf_{x_0\in\R^n} \norma{f-f^\sharp(\cdot +x_0)}_{L^1(\R^n)}^{ \theta_2} \leq ||v-u^{\sharp}||_{L^\infty(\Omega)}.
\end{equation}
Moreover, the dependence of $ C_1,\,  C_2$ and $  C_3$ from $\abs{\Omega}$ and $f^\sharp$ is explicit.
\end{teorema}

In order to prove Theorem \ref{main_theorem}, we show that the quantity $||v-u^{\sharp}||_\infty$  bounds from above each term on the left-hand side: this is the aim of Theorems \ref{asimmetryindex}, \ref{confrontol1} and \ref{confrontof}.

The first step  is contained in Theorem \ref{asimmetryindex}, which can be read in this sense: if the $L^\infty$ distance between $u^\sharp$ and $v$ is small, then the set $\Omega$ has a small Fraenkel asymmetry index, defined as

\begin{equation}\label{asimm}
	\alpha(\Omega):=\min_{x \in \R^{n}}\bigg \{  \dfrac{|\Omega\Delta B_r(x)|}{|B_r(x)|} \;,\; |B_r(x)|=|\Omega|\bigg \},
\end{equation}
where the symbol $\Delta$ stands for the symmetric difference and $|\Omega|$ is the measure of the set. 
This is a standard way to measure how much a set $\Omega$ differs from a ball $B_r$ with the same measure in the $L^1$ norm.

\begin{teorema}
\label{asimmetryindex}    
    Let $f\in L^{\frac{2n}{n+2}}(\Omega)$ if $n>2$, $f\in L^p(\Omega)$, $p>1$ if $n=2$ be a non-negative function and let $u$ and $v$ be the solutions to \eqref{mainproblem} and \eqref{symmetricproblem} respectively, and let $u^\sharp$ be the Schwarz rearrangement of $u$.
 Then, there exists a  constant $\widetilde C_1=\widetilde C_1(n)>0$, such that  
\begin{equation}\label{goal0}
    \alpha^3(\Omega)\leq \widetilde C_1\dfrac{\norma{v-u^{\sharp}}_\infty}{\abs{\Omega}^{\frac{2-n}{n}} \norma{f}_1},
        \end{equation}
where  
\begin{equation}\label{cappa1}
\widetilde C_1= \max\left\{\left(8 n^2\omega_n^{\frac{2}{n}}\right)^{3},(16n^2\omega_n^{\frac{2}{n}} \gamma_n)\right\},\end{equation}
  $\omega_n$ is the measure of the unit ball of $\R^n$ and $\gamma_n$ is the constant appearing in the quantitative isoperimetric inequality (see Theorem \ref{quant_isop_prop} ).
\end{teorema}
 
To prove this result we make use of a clever idea introduced in \cite{hansen}, used for obtaining non-sharp quantitative isoperimetric inequalities for the principal frequency and the capacity. 
This idea is also explained and exploited in the survey paper \cite{brasco} to prove a quantitative version of the P\'olya-Szeg\H o principle (see Theorem \ref{poliaszego}). Moreover,  we observe that Theorem \ref{asimmetryindex} implies the quantitative result contained in \cite{kim} in the case $f\equiv 1$.

In Theorem \ref{sopralivelli}, we obtain a similar bound for the Fraenkel asymmetry of almost any superlevel set $U_t=\{u>t\}$ of the function $u$ solution to \eqref{mainproblem}.

The second step for proving Theorem \ref{main_theorem} is contained in Theorem \ref{confrontol1}. We prove that the $L^1(\R^n)$ distance between $u$ and $u^\sharp$ is small, assuming that $u^\sharp$ is close to $v$ with respect to the $L^\infty$-distance.

\begin{teorema}\label{confrontol1} Let $f\in L^{\frac{2n}{n+2}}(\Omega)$ if $n>2$, $f\in L^p(\Omega)$, $p>1$ if $n=2$ be a non-negative function and let $u$ and $v$ be the solution to \eqref{mainproblem} and \eqref{symmetricproblem} respectively. Then, there exist two constants $\tilde{\theta}_1=\tilde{\theta}_1(n)>0$ and     $\widetilde C_2= \widetilde C_2(n)>0$ such that
\begin{equation}\label{esti_L1}
   \inf_{x_0\in \R^n} ||u-u^\sharp(\cdot +x_0)||_{L^1(\R^n)}\leq \widetilde C_2\dfrac{\norma{f}_{\frac{2n}{n+2}}^{1-2\tilde{\theta}_1}}{\norma{f}_{1}^{1-3\tilde{\theta}_1}}\abs{\Omega}^{1+\left(\frac{2}{n}-1\right)\tilde{\theta}_1}||v-u^{\sharp}||_{_{L^\infty(\Omega)}}^{\tilde{\theta}_1}.
\end{equation}
\end{teorema}
%. The Fraenkel asymmetry is an index of asymmetry, i.e. it measures how much a set differs from the ball in the $L^1$ norm  and it is defined as follows

In order to prove this result, we observe  in Lemma \ref{lem_grad} that a small $L^\infty$ distance between $u^\sharp$ and $v$ forces the P\'olya-Szeg\H o inequality, that is 
$$\int_{\Omega^\sharp} |\nabla u^\sharp|^2\, dx\le \int_\Omega \abs{\nabla u}^2\, dx,  
$$
to hold almost as an equality.
So, at this stage, it is natural to consider the P\'olya-Szeg\H o quantitative result proved in \cite{polyaquantitativa}:
\begin{equation*}
%\label{quanty_poly}
    \inf_{x_0 \in \R^n} \int_{\R^n} \abs{u(x)-u^\sharp(x+x_0)} \, dx \leq C \left[M_{u^\sharp}(E(u)^r)+E(u)\right]^s
\end{equation*}
where 
$$
E(u)= \frac{\displaystyle{\int_{\R^n} \abs{\nabla u}}}{\displaystyle{\int_{\R^n} \abs{\nabla u^{\sharp}}}}-1 \text{ and } M_{u^\sharp}(\delta)=\dfrac{\abs{\left\{|\nabla u^{\sharp}|\leq\delta\right\}\cap \left\{0<u^\sharp<\norma{u}_{\infty}\right\}}}{\abs{\{\abs{u}>0\}}}.
$$
Hence, to prove the closeness of $u$ to $u^\sharp$ 
it remains to bound from above these quantities in terms of $||v-u||_\infty$. 

Eventually, to conclude the proof of Theorem \ref{main_theorem}, we show the closeness of $f$ and $f^\sharp$, that is the third and last step, contained in Theorem \ref{confrontof}. Here, we prove that it is possible to bound the $L^m$ distance between $f$ and $f^\sharp$, for all $m<2$, provided $f\in L^2$.

%The last step is to prove that the function $f$ and $f^\sharp$ 

\begin{teorema}\label{confrontof}
    Let $f \in L^2(\Omega)$ be a non-negative function and let $u$, $v$ be the solutions to \eqref{mainproblem} and \eqref{symmetricproblem} respectively. Then, for  all $1\leq m <2$ there exist $\tilde{\theta}_2=\tilde{\theta}_2(n,m)>0$, $\tilde{\theta}_3=\tilde{\theta}_3(n,m)>0$  and a positive constant $\widetilde C_3= \widetilde C_3(n,m)>0$ such that
    $$\inf_{x_0\in\R^n}\norma{f-f^{\sharp}(\cdot + x_0)}_{L^m(\R^n)}\leq \widetilde C_3 \norma{f}^{\tilde{\theta}_2}_2 \norma{f}_1^{1-\tilde{\theta}_3-\tilde{\theta}_2}\abs{\Omega}^{\frac{\tilde{\theta}_2}{2}+\frac{1-m}{m}+\left(\frac{n-2}{n}\right)\tilde{\theta}_3}||v-u^\sharp||_{{L^\infty(\Omega)}}^{\tilde{\theta}_3} .$$
\end{teorema}

We prove this result %by combining Theorem \ref{asimmetryindex}, Theorem \ref{confrontol1} and 
using the quantitative version of the Hardy-Littlewood inequality (see \eqref{hardy_littlewood}) proved in \cite{cianchi_ferone}.% (see ).  
%\textcolor{magenta}{Spiegare bene le ipotesi su f}

The paper is organized as follows: in Section \ref{section2} we introduce the notation and the preliminary results that we will need throughout the paper;  Section \ref{section3} is devoted to the proof of Theorem \ref{asimmetryindex};  in Section \ref{section4}  we prove  Theorem \ref{confrontol1}, and in Section \ref{section5} we prove Theorem \ref{confrontof}. In Section \ref{section6} we prove Theorem \ref{main_theorem} and we collect a list of open problems. Finally, in Appendix \ref{appendix} we bound the asymmetry index $\alpha(\Omega)$ in terms of the difference of the $L^2$ norm of $v$ and $u$.
\section{Notation and preliminary results}
\label{section2}

Throughout this article, we will denote by $|\Omega|$ the Lebesgue measure of an open and bounded set of $\mathbb{R}^n$, with $n\geq 2$, and by $P(\Omega)$ the perimeter of $\Omega$ in the sense of De Giorgi (see \cite{ambrosio2000functions}).\\
For convenience, we will denote by $||\cdot||_p$ the norm $||\cdot||_{L^p(\Omega^\sharp)}$ or $||\cdot||_{L^p(\Omega)}$; otherwise, we will write the space explicitly.

If $\Omega$ is an open and Lipschitz set, it holds the following coarea formula. Some references for results relative to the sets of finite perimeter and the coarea formula are, for instance, \cite{ambrosio2000functions,maggi2012sets}.

 \begin{teorema}[Coarea formula]
 Let $f:\Omega\to\R$ be a Lipschitz function and let $u:\Omega\to\R$ be a measurable function. Then,
 \begin{equation}
   \label{coarea}
   {\displaystyle \int _{\Omega}u|\nabla f(x)|dx=\int _{\mathbb {R} }dt\int_{(\Omega\cap f^{-1}(t))}u(y)\, d\mathcal {H}^{1}(y)}.
 \end{equation}
 \end{teorema}

The rest of this Section is dedicated to stating some useful results on the rearrangement of functions, on quantitative inequalities and rescaling properties.
%Since we are assuming that $\partial \Omega$ is Lipschitz, we have that $P(\Omega)=\mathcal{H}^{n-1}(\partial\Omega)$, where $\mathcal{H}^{n-1}$ denotes the $(n-1)-$dimensional Hausdorff measure.

\subsection{Rearrangement of functions}
 For a general overview of this topic, we refer, for instance, to \cite{kes}.

 \begin{definizione}\label{distribution:function}
	Let $u: \Omega \to \R$ be a measurable function, the \emph{distribution function} of $u$ is the function $\mu : [0,+\infty[\, \to [0, +\infty[$ defined as the measure of the superlevel sets of $u$, i.e.
	$$
	\mu(t)= \abs{\Set{x \in \Omega \, :\,  \abs{u(x)} > t}}.
	$$
\end{definizione}

In particular, if $u$ is the solution to \eqref{mainproblem} and $v$ is the solution to \eqref{symmetricproblem}, we fix the following notations, for $t\geq0$
$$U_t=\left\lbrace x\in \Omega : u(x)>t\right\rbrace, \quad \mu(t)=\abs{U_t}$$
and%, if $v$ is the solution to \eqref{rob2}, using the same notations as above,  we set
$$V_t=\left\lbrace x\in \Omega^\sharp : v(x)> t\right\rbrace, \quad \nu(t)=\abs{V_t}.$$
By using the Coarea formula \eqref{coarea}, one can deduce the following expression for $\mu$

\begin{equation*}
    %\label{mugr0}
    \mu(t)=\abs{\{u>t\}\cap \{|\nabla u |=0\}}+ \int_t^{+\infty} \left(\int_{u=s}\frac{1}{\abs{\nabla u}}\, d\, \mathcal{H}^{n-1}\right)\, ds,
\end{equation*}

as a consequence, for almost all $t\in (0,+\infty)$,
    \begin{equation}\label{brothers1}
        \infty>-\mu'(t)\geq\displaystyle \int_{u=t}\dfrac{1}{|\nabla u|}d\mathcal{H}^{n-1}
    \end{equation}
and if $\mu$ is absolutely continuous, equality holds in \eqref{brothers1}.

\begin{definizione} \label{decreasing:rear}
	Let $u: \Omega \to \R$ be a measurable function, the \emph{decreasing rearrangement} of $u$, denoted by $u^\ast(\cdot)$,is defined as
$$u^*(s)=\inf\{t\geq 0:\mu(t)<s\}.$$
	\end{definizione}
 From Definitions \ref{distribution:function} and \ref{decreasing:rear}, one can prove that
	 $$u^\ast (\mu(t)) \leq t, \quad \forall t\ge 0,$$ 
  $$\mu (u^\ast(s)) \leq s \quad \forall s \ge 0.$$ 

	\begin{oss}\label{inverse}
	Let us notice that the function $\mu(\cdot)$ is decreasing and right continuous and the function $u^\ast(\cdot)$ is its generalized inverse.
	\end{oss}

	\begin{definizione}\label{rear}
	 The \emph{Schwartz rearrangement} of $u$ is the function $u^\sharp $ whose superlevel sets are balls with the same measure as the superlevel sets of $u$. 
	\end{definizione}
		We have the following relation between $u^\sharp$ and $u^*$:
	$$u^\sharp (x)= u^*(\omega_n\abs{x}^n),$$
 where $\omega_n$ is the measure of the unit ball in $\R^n$. One can easily check that the functions $u$, $u^*$ e $u^\sharp$ are equi-distributed, i.e. they have the same distribution function, and it holds
$$ \displaystyle{\norma{u}_{p}=\norma{u^*}_{p}=\lVert{u^\sharp}\rVert_{p}}, \quad \text{for all } p\ge1.$$
The following Lemma characterizes the absolute continuity of $\mu$ (we refer, for instance, to \cite{Brothers1988, cianchi_fusco}).

\begin{lemma}
    \label{asscontmu}
Let $u\in W^{1,p}(\R^n)$, with $p\in(1,+\infty)$. The distribution function $\mu$ of $u$ is absolutely continuous if and only if 
\begin{equation*}
    \abs{\{0<u<  ||u^\sharp||_\infty\}\cap \{|\nabla u^\sharp| =0\}}=0.
\end{equation*}

\end{lemma}

An important property of the decreasing rearrangement is the Hardy-Littlewood inequality, see \cite{hardy_classico}. 

\begin{teorema}[Hardy-Littlewood inequaliy]
Let us consider $h\in L^p(\Omega)$ and $g\in L^{p'}(\Omega)$. Then 
    \begin{equation}\label{hardy_littlewood}
 \int_{\Omega} \abs{h(x)g(x)} \, dx \le \int_{0}^{\abs{\Omega}} h^*(s) g^*(s) \, ds.
\end{equation}
\end{teorema}
By  choosing  $h=\chi_{\left\lbrace\abs{u}>t\right\rbrace}$ in \eqref{hardy_littlewood}, one has
\begin{equation*}
\int_{\abs{u}>t} \abs{h(x)} \, dx \le \int_{0}^{\mu(t)} h^*(s) \, ds.
\end{equation*}

If we require that the function $u$ is a \emph{Sobolev function}, i.e. $u\in W^{1,p}_0(\Omega)$, then also $u^\sharp$ is a Sobolev function and, moreover,  the gradient does not increase under symmetrization, as a consequence of the  P\'olya-Szegő inequality (see \cite{polya-szego}).
\begin{teorema}[P\'olya-Szegő inequality]
    \label{poliaszego}
    Let $u \in W^{1,p}_0(\Omega)$, then $u^{\sharp} \in W^{1,p}_0(\Omega^\sharp)$ and
	\begin{equation}\label{psi}
		\lVert \nabla u^{\sharp} \rVert_{p} \leq \norma{\nabla u}_{p}.
	\end{equation}
\end{teorema}

The following remark will be useful in the sequel.

\begin{oss}\label{prop:decreas}
  We explicitly observe  that the function 
 \[
\varphi:s\rightarrow \dfrac{1}{s}\int_0^sf^* 
 \]
 is decreasing.
 Indeed, let $s_1\leq s_2$, and let us  show that $\varphi(s_1)\geq\varphi(s_2).$
    This is true if and only if 
\[
\dfrac{1}{s_1}\int_0^{s_1}f^*-\dfrac{1}{s_2}\int_0^{s_1}f^*\geq \dfrac{1}{s_2}\biggl( \int_0^{s_2}f^*-\int_0^{s_1}f^*\biggr)
\]
\[
\dfrac{1}{s_1}\int_0^{s_1}f^*\geq   \frac{1}{s_2-s_1} \int_{s_1}^{s_2}f^*.
\]
Let us recall that $f^*$ is a decreasing function, so
\[
f^*(s)\geq f^*(s_1)\geq f^*(r)
\]
if $s<s_1<r$. Hence, 
\[
\dashint_0^{s_1}f^*(s)\;ds\geq f^*(s_1)\geq \dashint_{s_1}^{s_2}f^*(r)\;dr.
\]
\end{oss}

%\subsection{Some properties of the symmetric solution}

\vspace{2mm}
We recall now some properties of the symmetric solution.
The explicit expression of $v$, solution to problem \eqref{symmetricproblem},  is
\begin{equation}\label{v_explicit}
    v(x)=\int_{\omega_n \abs{x}^n}^{\abs{\Omega}} \dfrac{1}{n^2\omega_n^{\frac{2}{n}} t^{2-\frac{2}{n}}}\int_0^t f^*(r)\;dr\, dt,
\end{equation}
and its gradient is
\begin{equation}\label{grad_explicit}
    \abs{\nabla v}(x)=\dfrac{1}{n\omega_n\abs{x}^{n-1}}\int_{0}^{\omega_n\abs{x}^n}f^*(r)\;dr.
\end{equation}

From \eqref{grad_explicit} we can deduce that the gradient of $v$ vanishes at most at the origin and  Lemma \ref{asscontmu} ensures that the distribution function $\nu$ is absolutely continuous and it satisfies

\begin{equation}
    \label{nu}
    n^2\omega_n^\frac{2}{n} \nu(t)^{2-\frac{2}{n}} = (-\nu'(t))\left(\int_0^{\nu(t)}f^\ast \right).
\end{equation}
While for the distribution function $\mu$
\begin{equation}\label{mu}
 n^2\omega_n^{\frac{2}{n}}\leq \dfrac{-\mu'(t)}{\mu(t)^{2-\frac{2}{n}}}\int_0^{\mu(t)}f^*(r)\;dr,
    \end{equation}

As we want to obtain a quantitative version of \eqref{tal}, it is useful to recall briefly the outline of its proof.  
\begin{teorema}[Talenti, \cite{talenti76}]\label{talenti:originale}
 Let $f\in L^{\frac{2n}{n+2}}(\Omega)$ if $n>2$, $f\in L^p(\Omega)$, $p>1$ if $n=2$ and let $u$ and $v$ be the solutions to problems \eqref{mainproblem} and \eqref{symmetricproblem} respectively. Then, 
$$u^\sharp(x)\leq v(x).$$

\end{teorema} 

\begin{proof}
    Let us observe that $u$ is a weak solution to \eqref{mainproblem} if and only if 
\begin{equation}
    \label{wikdir}
    \int_{\Omega} \nabla u \nabla \varphi \, dx = \int_{\Omega} f \varphi \, dx, \qquad \forall \varphi \in W_0^{1,2}(\Omega).
\end{equation}
Let us fix $t\in [0,\norma{u}_\infty[$ and $h>0$, and let us define the function $\varphi$: 
\begin{equation}\label{phi}
	\left.
	\varphi (x)= 
	\right.
	\begin{cases}
	0 & \text{ if }u < t \\
	u-t & \text{ if }t< u< t+h \\
	h & \text{ if }u > t+h. 
	\end{cases} 
	\end{equation}
	 Using \eqref{phi}  as a test in the weak formulation \eqref{wikdir}, we have
 \begin{equation*}
    \int_{U_t\setminus U_{t+h}} \abs{\nabla u}^{2} \, dx=\int_{U_t\setminus U_{t+h}} f(x)(u-t) \, dx+ \int_{U_{t+h}} h f(x)\, dx.
    \end{equation*}
 If we divide now  by $h$, and we send $h$ to 0, we get 
\begin{equation*}\label{fmagg}
    \int_{\partial U_t} \abs{\nabla u} \, d\mathcal{H}^{n-1}(x)=\int_{U_t} f(x) \, dx\le \int_0^{\mu(t)} f^\ast(s)\, ds.
\end{equation*}

Applying  the isoperimetric inequality to the superlevel set $U_t$, the H\"older inequality and the Hardy-Littlewood inequality, we get, for almost every $t$,

\begin{align*}
	    n^2 \omega_n^{\frac{2}{n}}  \mu(t)^{2-\frac{2}{n}}& \leq P^2(U_t) \leq \left(\int_{\partial U_t} \,  d\mathcal{H}^{n-1}(x)\right)^2\\
	    &\leq\left(\int_{\partial U_t} \abs{\nabla u}\, d\mathcal{H}^{n-1}(x)\right) \left(\int_{\partial U_t}\frac{1}{\abs{\nabla u}} \, d\mathcal{H}^{n-1}(x)\right) \\
        &\leq \left(\int_0^{\mu(t)} f^\ast (s) \, ds\right) \left( -\mu'(t) \right).
	\end{align*}
Equivalently, we can write
\begin{equation*}
    1\le \frac{1}{n^2\omega_n^{\frac{2}{n}}}\mu(t)^{-2+\frac{2}{n}} \left(\int_0^{\mu(t)} f^\ast (s) \, ds\right) \left( -\mu'(t) \right),
\end{equation*}
and, integrating from $0$ to $t$, we obtain
\begin{equation}\label{tminv}
    t\le \frac{1}{n^2\omega_n^{\frac{2}{n}}}\int_{\mu(t)}^{\abs{\Omega}}r^{-2+\frac{2}{n}} \int_0^{r} f^\ast (s) \, ds \, dr= v^\ast(\mu(t)).
\end{equation}
 Since $u^\ast(r) =\inf\{\tau \,|\, \mu(\tau )<r\}$, \eqref{tminv} implies 
\begin{align*}
    u^\ast(r) =\inf\{\tau \,|\, \mu(\tau )<r\} \leq \inf  \left\{ v^\ast (\mu(\tau))\,|\, \mu(\tau )<r\right\} 
    \leq v^\ast(r),
\end{align*}
that gives \eqref{tal}.
\end{proof}

\begin{oss}\label{tal_grad}
 With the same techniques used for proving \eqref{tal}, it is possible to prove that 
 \begin{equation}
     \label{graduv}
     \int_\Omega \abs{\nabla u}^q\le   \int_{\Omega^\sharp} \abs{\nabla v}^q, \quad 0<q\le 2, 
 \end{equation}
and it is possible to estimate the $L^q$-norms of the gradient
of $v$ in terms of $L^p-$norms of $f$ via Bliss inequality (we refer to \cite[Equation (23.a)]{talenti76}), obtaining

\begin{equation}
    \label{stima:lq:gradv}
    \int_{\Omega^\sharp} \abs{\nabla v}^q=\frac{1}{n^q\omega_n^{\frac{q}{n}}} \int_0^{\abs{\Omega}}\left(\frac{1}{s}\int_0^2f^\ast\right)^q s^\frac{q}{n}\, ds \le K_1(n,q) \left(\int_0^{\abs{\Omega}}f^\frac{qn}{q+n}\right)^\frac{q+n}{n}.
\end{equation}
\end{oss}
\subsection{Some quantitative inequalities}

We recall now some quantitative results that we need to prove our main results Theorem \ref{asimmetryindex}, Theorem \ref{confrontol1} and Theorem \ref{confrontof}.
Let us start with the quantitative isoperimetric inequality, proved in \cite{fusco_maggi} (see also \cite{hall,halleco,cicaleseleonardi,Fuglede}).
\begin{teorema}
    \label{quant_isop_prop}
    There exists a constant $\gamma_n$ such that,  for any measurable set $\Omega$ of finite measure
    \begin{equation}\label{quant_isop}
        P(\Omega)\geq n\omega_n^{\frac{1}{n}}\abs{\Omega}^{\frac{n-1}{n}}\left(1+\dfrac{\alpha^2(\Omega)}{\gamma_n}\right),
    \end{equation}
    where $\alpha(\Omega)$ is defined in \eqref{asimm}.
\end{teorema}
We observe that the constant $\gamma_n$ is explicitly computed in \cite{figalli_maggi}.

In the proof of Theorem \ref{asimmetryindex}, we apply the quantitative isoperimetric inequality \eqref{quant_isop} to the superlevel set $U_t$ of the solution to \eqref{mainproblem}. In order to do that, it is useful to estimate the asymmetry of $U_t$ in terms of the asymmetry of $\Omega$, as in the following lemma (we refer to \cite[Lemma 2.8]{brasco}). 
\begin{lemma}\label{lembrasco}
    Let $\Omega\subset\mathbb{R}^n$ be an open set with finite measure and let $U\subset\Omega$, $\abs{U}>0$ be such that 
    \begin{equation}\label{cond}
        \dfrac{\abs{\Omega\setminus U}}{\abs{\Omega}}\leq \dfrac{1}{4}\alpha(\Omega).
        \end{equation}
                Then, we have
                \begin{equation}\label{claim_lem}
                    \alpha(U)\geq \dfrac{1}{2}\alpha(\Omega).
                \end{equation}
\end{lemma}

As already stressed in the introduction, a key role will be played by the quantitative version of P\'olya-Szeg\H o
inequality \eqref{poliaszego}, proved in \cite{polyaquantitativa}.
\begin{teorema}[Quantitative P\'olya-Szeg\H o
inequality]
\label{polya_quant}
    Let $u\in W^{1,2}(\R^n)$, $n\geq 2$. Then, there exist  positive constants $r$, $s$ and $C$, depending only on $n$,  such that, for every $u\in W^{1,2}(\R^n)$, it holds
\begin{equation}\label{quanty_poly}
   \inf_{x_0 \in \R^n} \dfrac{\displaystyle{\int_{\R^n} \abs{u(x)\pm u^\sharp(x+x_0)} \, dx }}{\abs{\{u>0\}}^{\frac{1}{n}+\frac{1}{2}}\norma{\nabla u^\sharp}_2}\leq C(n) \left[M_{u^\sharp}(E(u)^r)+E(u)\right]^s,
\end{equation}
where 
\begin{equation}\label{eumu}
E(u)= \frac{\displaystyle{\int_{\R^n} \abs{\nabla u}^2}}{\displaystyle{\int_{\R^n} |\nabla u^{\sharp}|^2}}-1 \qquad \text{ and } \qquad M_{u^\sharp}(\delta)=\dfrac{\abs{\left\{|\nabla u^{\sharp}|<\delta\right\}\cap \left\{0<u^\sharp<||u||_\infty\right\}}}{\abs{\{\abs{u}>0\}}}.
\end{equation}

\end{teorema}

We will also use a quantitative version of the Hardy-Littlewood inequality \eqref{hardy_littlewood}, proved in \cite{cianchi_ferone}. In order to state the result, we need to introduce the following notation:
\begin{align*}
    &\mu_h(t)=\abs{\{h>t\}}, \\
    &\mu_g(t)=\abs{\{g>t\}}, \\
    &g_h(x)=g^\ast(\mu_h(h(x))),
\end{align*}
and we observe that if the set $\Omega$ is a ball and the function $h$ is radial and decreasing, then the function $g_h(x)=g^\sharp(x)$.
Moreover, let us introduce the following Lorentz-type norm 
\begin{equation}\label{lorentz_cianchi}
\norma{g}_{\Lambda^q_p(\Omega)}=\left(\int_0^{\abs{\Omega}}g^\ast(s)^q \theta'_q(s)\;ds\right)^{1/q},
\end{equation}
where, for $s\in [0, \abs{\Omega})$,
\begin{equation}\label{theta}
    \theta_p(s)=\left(\int_0^s \left(-(h^\ast) '  (\sigma)\right)^{-1/(p-1)}  d\sigma \right)^{1/p'}.
\end{equation}
\begin{teorema}[Quantitative Hardy Littlewood inequality]\label{cianchi_hardy}
    Let  $h$ and $g$ be two measurable functions such that $hg\in L^1(\Omega)$ and let us assume that there exist  $p\in (1,+\infty)$ and $q\in [1,+\infty)$ such that $\theta_p(s)<\infty$ for every $0\leq s< |\Omega|$ and $\norma{g}_{\Lambda ^q_p(\Omega)}<\infty$. Then, by setting 
    \begin{equation}\label{cianchi_r}
        m=\dfrac{qp+1}{p+1},
    \end{equation}
    it holds
    \begin{equation}\label{cianchi_result}
        \int_{\Omega}h(x)g(x) \;dx+\dfrac{1}{2^{p+1}eq} \norma{g}^{-qp}_{\Lambda ^q_p(\Omega)}\norma{g-g_h}^{1+pq}_{L^m(\Omega)} \leq \int_0^{|\Omega|} h^\ast(s) g^\ast(s) \;ds.
    \end{equation}
\end{teorema}

\subsection{Rescaling properties}
\label{riscalamento}

    For simplicity, we will often prove our results under two extra assumptions:
    \begin{align}
     \label{om}   \abs{\Omega}=1,  \qquad \qquad \norma{f}_1=1. 
    \end{align}
    If this is not the case, we recover the result in the more general setting in the following way. Let us set 
    $$a=\dfrac{\abs{\Omega}^{1-\frac{2}{n}}}{\norma{f}_1}, \quad \quad b= \abs{\Omega}^{-\frac{1}{n}}.$$
    If $u$ and $v$ are solutions to \eqref{mainproblem} and \eqref{symmetricproblem} respectively, we define the functions $w(x)= au\left(\frac{x}{b}\right)$ and $z(x)= av\left(\frac{x}{b}\right)$, that are solutions to

    $$\begin{cases}
        -\Delta w=g & \text{in } \Tilde{\Omega}\\
        w=0 & \text{on } \partial\Tilde{\Omega},
    \end{cases} \qquad \qquad \begin{cases}
        -\Delta z=g^\sharp & \text{in } \Tilde{\Omega}^\sharp\\
        z=0 & \text{on } \partial\Tilde{\Omega}^\sharp,
    \end{cases}$$
   with $|\Tilde{\Omega}|=1$ and $\norma{g}_1=1$.  
Furthermore, the following holds
 \begin{align*}
    \alpha(\tilde \Omega) &= \alpha(\Omega); \\
      ||z-w^\sharp||_\infty &=\dfrac{\abs{\Omega}^{1-\frac{2}{n}}}{\norma{f}_1} ||v-u^\sharp||_\infty;\\
      ||w-w^\sharp||_1 &=\dfrac{\abs{\Omega}^{-\frac{2}{n}}}{\norma{f}_1} ||u-u^\sharp||_1;\\
      {\norma{g}}_p&= \abs{\Omega}^{1-\frac{1}{p}} \frac{\norma{f}_p}{\norma{f}_1} \,\,\, \forall p \,\text{ s.t. }\, f \in L^{p}(\Omega).
 \end{align*}
Moreover, if $\sigma(t)$ and $\mu(t)$ are the distribution functions of $w$ and $u$ respectively, then
$$
\sigma(t)= \abs{\Omega}^{-1}\mu\left(\frac{t}{a}\right).
$$

\section{Proof of Theorem \ref{asimmetryindex}}
\label{section3}

The following Lemma is inspired by Lemma $2.9$ (\textit{Boosted P\'olya-Szeg\H o principle}) contained in \cite{brasco}.
\begin{lemma} 
\label{boosted_prop}
 Let $f\in L^{\frac{2n}{n+2}}(\Omega)$ if $n>2$, $f\in L^p(\Omega)$, $p>1$ if $n=2$ and let $u$ and $v$ be the solutions to \eqref{mainproblem} and \eqref{symmetricproblem} respectively.

Let $s_{\Omega}$ be defined as follows
    \begin{equation}
        \label{essami}
        s_{\Omega}=\sup\left\{ t\geq 0: \mu(t)\geq \abs{\Omega}\left(1-\dfrac{\alpha(\Omega)}{4}\right)\right\} \in \R,
    \end{equation}
then, it results
\begin{equation}
        \label{boosted}
        \dfrac{1}{2\gamma_n}s_{\Omega}\alpha^2(\Omega)\leq ||v-u^{\sharp}||_{\infty},
    \end{equation}
    where $\gamma_n$ is the constant appearing in the quantitative isoperimetric inequality \eqref{quant_isop}.
\end{lemma}
Let us point out that $s_\Omega \in (0,+\infty)$:
\begin{itemize}
    \item It is greater then zero whenever $\alpha(\Omega)>0$ since $t \to \mu(t)$ is right-continuous and $\mu(0)= \abs{\Omega}$;
    \item It is finite since there exists $r$ small enough that $u^\ast(r) >s_\Omega$.  Indeed, by contradiction let us assume that  $\forall r \,\,u^\ast(r) <s_\Omega,$ then $\forall r $
$$
r \geq \mu(u^\ast(r)) > \abs{\Omega} \left(1 - \frac{\alpha(\Omega)}{4}\right).
$$
This is a contradiction for $r$ sufficiently small.
\item By the definition of decreasing rearrangement \ref{decreasing:rear} it follows that
$$s_\Omega=u^\ast \left(\abs{\Omega}\left(1-\frac{\alpha(\Omega)}{4}\right)\right)$$
\end{itemize}

\begin{proof} We  assume that $\alpha(\Omega)>0$, otherwise the result is trivial. 
    Let us apply the quantitative isoperimetric inequality \eqref{quant_isop} to the level set $U_t$ and, then, argue as in the proof of Theorem \ref{talenti:originale},

    \begin{equation}\label{peromega}
        \begin{aligned}
       n^2\omega_n^{\frac{2}{n}}\mu^{2-\frac{2}{n}}(t)\left(1+\dfrac{2}{\gamma_n}\alpha^2(U_t)\right)&\le n^2\omega_n^{\frac{2}{n}}\mu^{2-\frac{2}{n}}(t)\left(1+\dfrac{1}{\gamma_n}\alpha^2(U_t)\right)^2\leq P^2(U_t) \\
       &\leq \int_{u=t}\abs{\nabla u}\int_{u=t}\dfrac{1}{\abs{\nabla u}}\le
      \left( \int_0^{\mu(t)}f^*\right) (-\mu'(t)).
        \end{aligned}
    \end{equation}
    If we divide by $n^2\omega_n^{\frac{2}{n} }\mu^{2-\frac{2}{n}}(t)$ and we integrate between $0$ and $\tau$, we get

    \begin{equation*}
        \tau+\dfrac{2}{\gamma_n}\int_0^{\tau}\alpha^2(U_t)\;dt\leq \int_{\mu(\tau)}^{\abs{\Omega}}\left(\dfrac{1}{n^2\omega_n^{\frac{2}{n}} t^{2-\frac{2}{n}}}\int_0^tf^*\right)\, dt=v^\ast (\mu(\tau)).
    \end{equation*}    
    
    The last, due to the definition of decreasing rearrangement, can be written as
    \begin{equation}\label{fine}
        \dfrac{2}{\gamma_n}\int_0^{u^*(r)}\alpha^2(U_t)\;dt\leq\int_r^{\abs{\Omega}}\left(\dfrac{1}{n^2\omega_n^{\frac{2}{n}}t^{2-\frac{2}{n}}}\int_0^tf^*\right)\, dt-u^*(r)= v(r)-u^*(r).
    \end{equation}
   \begin{comment}
        Indeed, the last implies
$$
 \tau \leq  v^\ast (\mu(\tau)) - \dfrac{2}{\gamma_n}\int_0^{u^\ast(\mu (\tau))}\alpha^2(U_t)\;dt,
$$
and since $u^\ast(r) =\inf\{\tau |\, \mu(\tau )<r\}$
\begin{align*}
    u^\ast(r)& =\inf\{\tau |\, \mu(\tau )<r\} \\ &\leq \inf  \left\{ v^\ast (\mu(\tau)) - \dfrac{2}{\gamma_n}\int_0^{u^\ast(\mu (\tau))}\alpha^2(U_t)\;dt|\, \mu(\tau )<r\right\} \\
    & \leq \inf  \left\{ v^\ast (\mu(\tau))|\, \mu(\tau )<r\right\} - \dfrac{2}{\gamma_n}\int_0^{u^\ast(r)}\alpha^2(U_t)\;dt\\
    & \leq v^\ast(r) - \dfrac{2}{\gamma_n}\int_0^{u^\ast(r)}\alpha^2(U_t)\;dt.
\end{align*}
  \end{comment}
We define the set
     \begin{equation}\label{A}
    A:= \left\{   
 t\geq 0: \mu(t)\geq \abs{\Omega}\left(1-\dfrac{\alpha(\Omega)}{4}\right)  \right\},\end{equation}
then $0 \in A$  whenever $\alpha(\Omega)>0$, and it is an interval since $\mu$ is a decreasing function.  Moreover, for all $t\in A$, it holds

 \begin{equation*}
     \frac{\abs{\Omega \setminus U_t}}{|\Omega|}=1-\frac{\mu(t)}{\abs{\Omega}}\le \frac{1}{4} \alpha(\Omega).
 \end{equation*}
 Therefore, we can apply Lemma \ref{lembrasco} and we get \begin{equation}
     \label{conti}
     \alpha(U_t)\geq \frac{\alpha(\Omega)}{2}\qquad \forall t \in A.
 \end{equation}
 On the other hand, from \eqref{fine} it follows
 \begin{equation}\label{questa}
     \min\left\{u^*(r),s_{\Omega}\right\}\dfrac{\alpha^2(\Omega)}{2\gamma_n}\leq \dfrac{2}{\gamma_n}\int_0^{u^*(r)}\alpha^2(U_t)\;dt\leq v(r)-u^*(r), \quad \forall r\in[0, \abs{\Omega}],
 \end{equation}
and, 
\begin{comment}
    P.A. $\not\exists r $ s.t. $$u^\ast(r) >s_\Omega$$. Then $\forall r $
$$
r \geq \mu(u^\ast(r)) > \abs{\Omega} \left(1 - \frac{\alpha(\Omega)}{4}\right).
$$
It is a contradiction for $r$ sufficiently small.
\end{comment}
choosing $r$ sufficiently small in \eqref{questa}, we can conclude that
\begin{equation*}
  \dfrac{1}{2\gamma_n}  s_{\Omega}\alpha^2(\Omega)\leq  ||v-u^\sharp||_{\infty}.
\end{equation*}

\end{proof}
\begin{comment}
\begin{teorema}
\label{asimmetryindex}    
    Let $u$ and $v$ the solutions to \eqref{mainproblem} and \eqref{symmetricproblem} respectively, and let $u^\sharp$ be the Schwarz rearrangement of $u$.
 Then, there exists a constant $K_1=K_1(n)$, such that  
\begin{equation}\label{goal0}
    \alpha^3(\Omega)\leq K_1\dfrac{\norma{v-u^{\sharp}}_\infty}{\abs{\Omega}^{\frac{2-n}{n}} \norma{f}_1},
        \end{equation}
where  
\begin{equation}\label{cappa1}
K_1= \max\left\{\left(8 n^2\omega_n^{\frac{2}{n}}\right)^{3},(2 \gamma_n)^\frac{3}{2}\right\}.\end{equation}
\end{teorema}
\end{comment}

We are now in a position to prove Theorem \ref{asimmetryindex}.

\begin{proof}[Proof Theorem \ref{asimmetryindex}]
    Firstly, let us suppose $\abs{\Omega}=1$, and $\norma{f}_{1}=1$. Let us assume that $\alpha(\Omega)>0$, otherwise the inequality \eqref{goal0} is trivial. %Then,
   % $s_{\Omega}$ defined in \eqref{essami} is strictly positive.
We set
     \begin{equation}\label{eps}
        \varepsilon:=||v-u^{\sharp}||_\infty,
    \end{equation}
   and we suppose that  $\varepsilon<1$.  If this is not the case, inequality \eqref{goal0} holds, being $\alpha(\Omega)<2$ and $ K_1\geq 8$.
    Definition \eqref{eps} implies
    \begin{equation*}
      v(x)-\varepsilon  \leq u^{\sharp}(x)\leq v(x), \qquad \text{for almost any } x\in \Omega^\sharp
    \end{equation*}
  and it follows that 
\begin{equation}\label{distr}
\nu(t+\varepsilon)\leq \mu(t) \leq \nu(t).
\end{equation}

Moreover, as $\nu$ is absolutely continuous, 
%and recalling \eqref{nu_derivata},
we have
\begin{equation*}
\begin{aligned}
    \nu(t+\varepsilon)-\nu(t)&=\int_t^{t+\varepsilon}\nu'=\int_t^{t+\varepsilon}\dfrac{-n^2\omega_n^{\frac{2}{n}}}{\nu^{-1+\frac{2}{n}}\dashint_0^{\nu(r)}f^*}\\
    &\geq\int_t^{t+\varepsilon}- \dfrac{n^2\omega_n^{\frac{2}{n}}}{\abs{\Omega}^{-1+\frac{2}{n}}}\dfrac{\abs{\Omega}}{\norma{f}_1}=- n^2\omega_n^{\frac{2}{n}}\varepsilon,
    \end{aligned}
\end{equation*}
where the inequality follows from Remark \ref{prop:decreas}.
So, \eqref{distr} becomes
\begin{equation}
    \label{muquasinu}
\nu(t)-n^2\omega_n^{\frac{2}{n}}\varepsilon\leq \mu(t) \leq \nu(t).
\end{equation}
By the definition of $s_\Omega$ and the property of decreasing rearrangement, we get
\begin{equation}
\label{muustar}
\mu(s_\Omega)=\mu\left(u^\ast\left(1-\frac{\alpha(\Omega)}{4}\right)\right)   \le 1-\frac{\alpha(\Omega)}{4}.
\end{equation}
So, combining \eqref{muquasinu}, \eqref{muustar} and  the absolutely continuity 
 of $\nu(t)$, we have
\begin{equation*}
    \begin{aligned}
    \frac{\alpha(\Omega)}{4}&\le 1-\mu(s_\Omega)\le 1- \nu(s_\Omega) + n^2 \omega_n^{\frac{2}{n}}\varepsilon
    =\int_0^{s_{\Omega}}-\nu'(t)dt +n^2 \omega_n^{\frac{2}{n}}\varepsilon\\&=\int_0^{s_{\Omega}} \dfrac{n^2\omega_n^{\frac{2}{n}}}{\nu^{-1+\frac{2}{n}}\dashint_0^{\nu(r)}f^*}\;dr +n^2 \omega_n^{\frac{2}{n}}\varepsilon\leq n^2\omega_n^{\frac{2}{n}}\abs{\Omega}^{2-\frac{2}{n}}\dfrac{s_{\Omega}}{\norma{f}_1} +n^2 \omega_n^{\frac{2}{n}}=n^2\omega_n^\frac{2}{n}(s_\Omega+\varepsilon),
    \end{aligned}
\end{equation*}
that gives 
\begin{equation}
    \label{veditu}
    s_\Omega\ge \frac{\alpha(\Omega)}{4n^2\omega_n^{\frac{2}{n}}}-\varepsilon.
\end{equation}
We now distinguish two cases
\begin{description}
     \item[Case 1] Let us assume that $\displaystyle{n^2\omega_n^\frac{2}{n}\varepsilon \le \frac{\alpha(\Omega)}{8}}$, then by \eqref{veditu}
 and Lemma \ref{boosted_prop}, we obtain
 \begin{equation} 
 \label{caso1_quant}
     \frac{1}{16 n^2 \omega_n^{\frac{2}{n}}\gamma_n} \alpha^3(\Omega)\le \varepsilon;
 \end{equation}

     \item[Case 2] Let us assume that $\displaystyle{n^2\omega_n^\frac{2}{n}\varepsilon > \frac{\alpha(\Omega)}{8}}$
     then

     \begin{equation}
     \label{caso2_quant}
         \alpha(\Omega) \le 8 n^2\omega_n^\frac{2}{n}\varepsilon \le 8n^2\omega_n^\frac{2}{n}\varepsilon^\frac{1}{3}
     \end{equation}
 \end{description}
 From  \eqref{caso1_quant} and \eqref{caso2_quant}, we get the claim \eqref{goal0} by choosing
 $$\widetilde C_1= \max\left\{\left(8 n^2\omega_n^{\frac{2}{n}}\right)^{3},(16n^2\omega_n^{\frac{2}{n}} \gamma_n)\right\},$$
in the case $\abs{\Omega}=1$ and $||f||_1=1$. In the general case, using the scaling properties contained in Section  \ref{riscalamento}, we obtain the desired claim \eqref{goal0}.

\end{proof}

An equivalent result holds for almost every superlevel set $U_t=\{u>t\}$.

\begin{teorema}
\label{sopralivelli}
 Let $u$ and $v$ be the solutions to \eqref{mainproblem} and \eqref{symmetricproblem} respectively, and let $u^\sharp$ be the Schwarz rearrangement of $u$.
 Then, it holds
    \begin{equation}   \label{goal1}
        \alpha^3(U_t)\leq \widetilde C_1 \dfrac{\abs{\Omega}^3}{\mu^3(t)}\frac{\abs{\Omega}^{1-\frac{2}{n}}}{\norma{f}_1} ||v-u^{\sharp}||_{\infty},
    \end{equation}
    where $\widetilde C_1$ is defined in \eqref{cappa1}.
\end{teorema}

\begin{proof}
Firstly, let us assume that $\abs{\Omega}=1$, and $\norma{f}_{1}=1$. 

As in the proof of Theorem \ref{asimmetryindex}, let us suppose that $\alpha(U_t)>0$ and $||v-u^\sharp||_\infty<1$, otherwise the inequality \eqref{goal1} is trivial.
By applying the quantitative isoperimetric inequality \eqref{quant_isop} to the superlevel sets $U_t$, and arguing as  in \eqref{peromega}, we have 
\begin{equation}
    1+\frac{2}{\gamma_n}\alpha^2(U_t)\leq \dfrac{-\mu'}{n^2\omega_n^{\frac{2}{n}} \mu(t)^{2-\frac{2}{n}}} \int_0^{\mu(t)}f^*
\end{equation}
and, integrating from $t$ to $\tau>t$, we obtain
$$ \tau -t+    \frac{2}{\gamma_n}\int_t^{\tau} \alpha^2(U_r)\;dr\leq \int_t^{\tau}\dfrac{-\mu'(r)}{n^2 \omega_n^{\frac{2}{n}}\mu(r)^{2-\frac{2}{n}}}\int_0^{\mu(r)}f^*\;dr.
$$
Performing the change of variables $\mu(r)=s$ on the right-hand side, we have
$$
\tau-t+\dfrac{2}{\gamma_n}\int_t^{\tau}\alpha^2(U_r)\;dr\leq \int_{\mu(\tau)}^{\mu(t)}\dfrac{1}{n^2\omega_n^{\frac{2}{n}} s^{2-\frac{2}{n}}}\int_0^{s}f^*\;ds=v^*(\mu(\tau))-v^*(\mu(t)),
$$
that implies
\begin{equation}
\label{uno}
\dfrac{2}{\gamma_n}\int_t^{u^\ast(r)}\alpha^2(U_r)\;dr\leq v^*(r)-u^\ast(r)\leq ||v-u^\sharp||_\infty,
\end{equation}
obtained using  \eqref{tminv} and the definition of decreasing rearrangement.

Arguing  as in Lemma \ref{boosted_prop}, we define the set
 \begin{equation}\label{A_t}
    A_t:= \left\{   
 r\geq t: \mu(r)\geq \mu(t)\left(1-\dfrac{\alpha(U_t)}{4}\right)  \right\},
 \end{equation}
that is non-empty whenever $\alpha(U_t)>0$.
Moreover,  we define $ s_t=\sup(A_t)$ and we observe that 
\begin{itemize}
    \item $s_t>t$, as we are assuming $\alpha(U_t)>0$;
    \item $s_t$ is finite;
    \item $\displaystyle{s_t=u^\ast \left(\mu(t)\left(1-\frac{\alpha(U_t)}{4}\right)\right)};$
    \item for all $r<s_t$, the set $U_r$ verifies the hypothesis \eqref{cond} and, so,  
\end{itemize}

\begin{equation}\label{due}
     \alpha(U_r)\geq \dfrac{\alpha(U_t)}{2}.
 \end{equation}
Combining \eqref{uno} and \eqref{due} we have 
\begin{equation}\label{tre}
    \displaystyle{\dfrac{1}{2\gamma_n}(s_t-t)\alpha^2(U_t)\leq \varepsilon}.
\end{equation}

Let us observe that

$$\mu(s_t)=\mu\left(u^\ast \left(\mu(t)\left(1-\frac{\alpha(U_t)}{4}\right)\right)\right)\le \mu(t)\left(1-\frac{\alpha(U_t)}{4}\right) ,$$
so, we have

\begin{equation*}
\begin{aligned}
    \mu(t) \frac{\alpha(U_t)}{4} &\le \mu(t)-\mu(s_t)\le \nu(t)-\nu(s_t) +n^2\omega_n^\frac{2}{n}\varepsilon=\int_t^{s_t} -\nu'(r)\,dr +n^2\omega_n^\frac{2}{n}\varepsilon\\
    &=\int_t^{s_t} \dfrac{n^2\omega_n^{\frac{2}{n}}}{\nu^{-1+\frac{2}{n}}\dashint_0^{\nu(r)}f^*}\;dr+n^2\omega_n^\frac{2}{n}\varepsilon\le \dfrac{n^2\omega_n^{\frac{2}{n}}\abs{\Omega}^{2-\frac{2}{n}}}{\norma{f}_1}(s_t-t)+ n^2\omega_n^{\frac{2}{n}}\varepsilon,
\end{aligned}
\end{equation*}
that gives
\begin{equation}
    \label{veditu2}
    (s_t-t)\ge \mu(t)\frac{\alpha(U_t)}{4n^2\omega_n^{\frac{2}{n}}}-\varepsilon.
\end{equation}

Now,  we distinguish two cases. 
\begin{description}
     \item[Case 1] Let us assume that $\displaystyle{n^2\omega_n^\frac{2}{n}\varepsilon \le \mu(t)\frac{\alpha(U_t)}{8}}$, then by \eqref{veditu2}
 and \eqref{tre}, we obtain
 \begin{equation} 
 \label{caso1t}
     \frac{\mu(t)}{16 n^2 \omega_n^{\frac{2}{n}}\gamma_n} \alpha^3(U_t)\le \varepsilon;
 \end{equation}

     \item[Case 2] Let us assume that $\displaystyle{n^2\omega_n^\frac{2}{n}\varepsilon > \mu(t)\frac{\alpha(U_t)}{8}}$
     then

     \begin{equation}
     \label{caso2t}
         \mu(t)\alpha(U_t) \le 8 n^2\omega_n^\frac{2}{n}\varepsilon \le 8n^2\omega_n^\frac{2}{n}\varepsilon^\frac{1}{3}
     \end{equation}
 \end{description}
From \eqref{caso1t} and \eqref{caso2t}, we get the claim \eqref{goal1} with the same costant $\widetilde C_1 $ defined in \eqref{cappa1}.

As in Theorem \ref{asimmetryindex}, we prove the Theorem in the case $\abs{\Omega}=1$ and $||f||_1=1$, and we recover the general case using the rescaling properties from Section \ref{riscalamento}.
    \end{proof}

\section{Proof of the Theorem \ref{confrontol1}}
\label{section4}

The aim of this Section is to prove Theorem \ref{confrontol1}, and 
as a first step, we observe that the following lemma holds.
\begin{lemma}
\label{lem_grad}
     Let $f\in L^{\frac{2n}{n+2}}(\Omega)$ if $n>2$, $f\in L^p(\Omega)$, $p>1$ if $n=2$ and let $u$ be the solution to  \eqref{mainproblem} and let $v$ be the solution to \eqref{symmetricproblem}, then 
\begin{equation}
\label{tops}
\int_{\Omega}\abs{\nabla u}^2-\int_{\Omega^\sharp}|\nabla u^\sharp|^2\leq 2\norma{f}_{1} ||v-u^\sharp||_\infty.
\end{equation}
\end{lemma}
\begin{proof}
We observe that   the solution $v$ to \eqref{symmetricproblem} is the unique minimizer in $W^{1,2}_0(\Omega^\sharp)$ of the functional
$$
\mathcal{F}(\varphi)=\dfrac{1}{2}\int_{\Omega^{\sharp}}\abs{\nabla \varphi}^2-\int_{\Omega^\sharp}f^\sharp \varphi,
$$
that implies %, by using the P\'olya-Szegő inequality \eqref{psi} and the Hardy-Littlewood inequality \eqref{hardy_littlewood}, we get 
\begin{equation*}
\label{minfun}
\dfrac{1}{2}\int_{\Omega^{\sharp}}\,|\nabla u^\sharp|^2-\int_{\Omega^\sharp}f^\sharp u^\sharp\geq \dfrac{1}{2}\int_{\Omega^{\sharp}}\abs{\nabla v}^2-\int_{\Omega^\sharp}f^\sharp v.
\end{equation*}
Hence, by the Talenti's inequality (see Remark \ref{tal_grad} with $q=2$), %i.e.

$$  \int_{\Omega}\abs{\nabla u}^2\le \int_{\Omega^\sharp}\abs{\nabla v}^2,$$ 
and the H\"older inequality, we obtain
\begin{equation*}
    \label{topolyaszego}
\int_{\Omega}\abs{\nabla u}^2-\int_{\Omega^\sharp}|\nabla u^\sharp|^2 \leq \int_{\Omega^\sharp}\abs{\nabla v}^2-\int_{\Omega^\sharp}|\nabla u^\sharp|^2\leq 2\int_{\Omega^\sharp}f^\sharp(v-u^\sharp)<2 \norma{f}_1||v- u^\sharp||_\infty, 
\end{equation*}
and  the thesis \eqref{tops} follows. 
\end{proof}
In particular, \eqref{tops} suggests that the P\'olya-Szeg\H o inequality holds true almost as an equality when $||v-u^{\sharp}||_{\infty}$ is small enough. So it is natural to consider the P\'olya-Szeg\H o quantitative inequality, recalled in Theorem \ref{polya_quant}.
%$$
%\inf_{x_0 \in \R^n} \int_{\R^n} \abs{u(x)-u^\sharp(x+x_0)} \, dx \leq C \left[M_{u^\sharp}(E(u)^r)+E(u)\right]^s
%$$
%where 
%$$
%E(u)= \frac{\int_{\R^n} \abs{\nabla u}}{\int_{\R^n} \abs{\nabla u^{\sharp}}}-1 \text{ and } M_{u^\sharp}(\delta)=\dfrac{\abs{\left\{\abs{\nabla u^{\sharp}}<\delta\right\}\cap \left\{0<u^\sharp<\norma{u}_{\infty}\right\}}}{\abs{\Omega}}.
%$$

At this point, to prove the closeness of $u$ to $u^\sharp$ it is enough to show that, for every positive $\delta$, the quantity $M_{u^{\sharp}}(\delta)$ can be estimated from above by some power of $\delta$. This is not true in general for a generic Sobolev function, but it holds for $u$ because it is the solution of \eqref{mainproblem}.

The main idea to bound from above the quantity $M_{u^{\sharp}}(\delta)$ is to write the set $\{\abs{\nabla u} \leq \delta\}$ as the union of suitable subsets, and to bound from above the measure of each of them. In order to do that, in Lemma \ref{lemmaI}
 we define a set $I\subseteq [0, \norma{u}_\infty]$ and, in \eqref{t_epsilon}, we define a positive number $t_{\varepsilon,\beta}\in{[0, \norma{u}_\infty]}$, so that we have
\begin{align*}
   \{\nabla u^\sharp < \delta\}
    =&\Big(\left\{|\nabla u^\sharp|<\delta\right\}\cap u^{\sharp-1}(I^c\cap (0,t_{\varepsilon,\beta})\Big)\\
    &\cup\Big(\left\{|\nabla u^\sharp|<\delta\right\}\cap u^{\sharp-1}(I\cap (0,t_{\varepsilon,\beta})\Big)\\
    &\cup\Big(\left\{|\nabla u^\sharp|<\delta\right\}\cap u^{\sharp-1}(t_{\varepsilon,\beta}, +\infty)\Big).
\end{align*}

In Propositions \ref{prop1} and \ref{prop2} and in Remark \ref{rem3}, we study separately the above sets.% while in Proposition \ref{Musharp} and in Theorem \ref{confrontof} we conclude.

We will prove  Lemma \ref{lemmaI}, Propositions \ref{delta0}, \ref{prop1}, \ref{prop2}, \ref{Musharp} under the additional assumptions
$$\norma{f}_1=1, \quad \abs{\Omega}=1, \quad ||{v-u^\sharp}||_\infty \le \varepsilon_0,$$
where $\varepsilon_0$ will be suitably chosen later.

\begin{lemma} \label{lemmaI}
Let us fix $\alpha>0$, let $\varepsilon=||v-u^\sharp||_\infty$ and let us define 
the set $I$ as follows
   \begin{equation}
\label{I}
I=\left\{t \in [0, \norma{u}_\infty] :\,\int_{v=t}\abs{\nabla v}-\int_{u^\sharp=t}\abs{\nabla u^\sharp}>\varepsilon^\alpha\right\}.
\end{equation}
    Then,    
    \begin{equation}\label{I_bound}
        \abs{I}\le 2  \varepsilon^{1-\alpha},
    \end{equation}
and, for every  $t\in I^c$,  it holds
\begin{equation}
\label{gradient}
\int_{v=t}\abs{\nabla v}\;-\int_{u^\sharp=t}|\nabla u^\sharp|\leq \varepsilon^{\alpha}.
\end{equation}
    
\end{lemma}
\begin{proof}
Claim \eqref{gradient}  is a direct consequence of the definition of the set $I$ in \eqref{I}. Moreover, using  the Coarea formula \eqref{coarea}, we have
\begin{align*}
    \varepsilon^{\alpha}\abs{I}\leq \int_{I}\biggl(\int_{v=t}\abs{\nabla v}-\int_{\{u^\sharp=t\}}|\nabla u^\sharp|\biggr)&\leq \int_0^{+\infty}\biggl(\int_{v=t}\abs{\nabla v}-\int_{u^\sharp=t}|\nabla u^\sharp|\biggr)\\&=\int_{\Omega^\sharp}\abs{\nabla v}^2-\int_{\Omega^\sharp}|\nabla u^\sharp|^2<2 \varepsilon, 
    \end{align*}
and claim \eqref{I_bound} follows.
\end{proof}

\begin{prop}
\label{delta0}
Let $\delta>0$ and let  $v$ be the solution to \eqref{symmetricproblem}. Then, there exists a positive constant $K_2=K_2(n)$ such that
   \begin{equation}\label{bound_grad}
      |\{ \abs{\nabla v}\le\delta\}|\leq K_2\delta^n,
   \end{equation}
   where $$K_2=\omega_n n^n.$$
\end{prop}
\begin{proof}
To prove \eqref{bound_grad}, it is enough to show
\begin{equation}
\label{delta1}
\left\{\abs{\nabla v}\le\delta\right\}\subseteq \left\{\abs{x}<n\delta\right\}. 
\end{equation}
So, if we take $x\in \{\abs{\nabla v}\le\delta\}$, we have, using \eqref{v_explicit}, \eqref{grad_explicit} and Remark \ref{prop:decreas}
$$\delta\ge\abs{\nabla v}(x)=\dfrac{1}{n \omega_n\abs{x}^{n-1}}\int_0^{\omega_n\abs{x}^n}f^*(r)\;dr\geq \dfrac{\abs{x}}{n\abs{\Omega}}\int_0^{\abs{\Omega}}f^*(r)\;dr=\frac{\abs{x}}{n}, $$
and, as a consequence, we have
\begin{equation}
    \label{constant_K_1}
    \abs{\nabla v}(x)\geq \dfrac{\abs{x}}{n},
\end{equation}
and we can conclude.
\end{proof}

\begin{prop}\label{prop1}
Let $u$ be the solution to \eqref{mainproblem}, let $u^\sharp$ be its Schwartz rearrangement and $I$ as in \eqref{I}. If we define the following quantity
\begin{equation}
    \label{t_epsilon}
    t_{\varepsilon,\beta}=\sup\left\{t>0\;:\;\;\mu(t)>\varepsilon^{n'\beta}\right\}, \quad n'=\frac{n}{n-1 }, 
\end{equation}
then, for every $\delta>0$, it holds 
    \begin{equation}\label{claim4}
       \abs{\left\{|\nabla u^\sharp|<\delta\right\}\cap u^{\sharp-1}(I^c\cap (0,t_{\varepsilon,\beta}))}\leq\dfrac{1}{n^n}\biggl(\delta+\dfrac{\varepsilon^{\alpha-\beta}}{n\omega_n^{\frac{1}{n}}}\biggr)^n,
   \end{equation}
  for every $\alpha,\beta\in \mathbb{R}$ such that  $0<\beta<\alpha<1$, $\beta<\dfrac{1}{n'}$.
\end{prop}
\begin{proof}

We observe that \eqref{gradient} in Lemma \ref{lemmaI} is equivalent to
    \begin{equation}\label{gradv2}
        P(v=t)\abs{\nabla v}(y)-P(u^\sharp=t)|\nabla u^\sharp|(x)<\varepsilon^{\alpha}, 
    \end{equation}
    for every $y\in \{v=t\}$ and $x\in \{u^\sharp = t\}$. If we fix  $x\in \{u^\sharp = t\}$ and we consider  $y=\ell(t)x$, where 
$\ell(t)=\left(\frac{\nu(t)}{\mu(t)}\right)^{\frac{1}{n}}$, we have that $y\in \{v=t\}$.

 Then, for every $t\in I^c$ such that $t< t_{\varepsilon,\beta}$, \eqref{gradv2} becomes
    \begin{equation}
    \label{gradv3}
        P(v=t)\abs{\nabla v}(\ell(t)x)-P(u^\sharp=t)|\nabla u^\sharp|(x)<\varepsilon^{\alpha}, 
    \end{equation}
and, since 
    $$P(u^\sharp=t)\leq P(v=t),$$
    we have 
    \begin{equation}
    \label{gradv4}
    \abs{\nabla v}(\ell(t)x)-|\nabla u^\sharp|(x)\leq \dfrac{\varepsilon^\alpha}{P(v=t)}.
    \end{equation}
Now, from definition \eqref{t_epsilon}, follows 
\begin{equation}
\label{perimetro_sottone}
\omega_n\biggl(\dfrac{P(v=t)}{n\omega_n}\biggr)^{n'}=\nu(t)>\mu(t)>\varepsilon^{n'\beta},
\end{equation}
and, combining \eqref{gradv4} with \eqref{perimetro_sottone}, we obtain 
\begin{equation}
\label{diff_grad1}
\abs{\nabla v}(\ell(t)x)-|\nabla u^\sharp|(x)\leq \dfrac{\varepsilon^{\alpha-\beta}}{n \omega_n^{\frac{1}{n}}}.
\end{equation}
Moreover, combining \eqref{constant_K_1} and \eqref{diff_grad1} and the fact that $\ell(t)\geq 1$, we have 
\begin{equation}
\label{diff_grad2}
    \dfrac{\abs{x}}{n}  -|\nabla u^\sharp|(x) \leq \abs{\nabla v}(\ell(t)x)-|\nabla u^\sharp|(x)\leq \dfrac{\varepsilon^{\alpha-\beta}}{n \omega_n^{\frac{1}{n}}}, \qquad \forall x \in u^{\sharp-1}(I^c \cap (0, t_{\varepsilon,\beta})).
\end{equation}
From \eqref{diff_grad2}, follows that 
   \begin{align*}
     \left\{|\nabla u^\sharp|<\delta\right\}\cap u^{\sharp-1}(I^c\cap (0,t_{\varepsilon,\beta}))\subseteq \left\{\abs{x}<n\delta+\dfrac{\varepsilon^{\alpha-\beta}}{ \omega_n^{\frac{1}{n}}} \right\},
   \end{align*}
and, consequently,
\begin{equation*}
    \abs{\left\{|\nabla u^\sharp|<\delta\right\}\cap u^{\sharp-1}(I^c\cap (0,t_{\varepsilon,\beta})}\leq\omega_n\biggl(n\delta+\dfrac{\varepsilon^{\alpha-\beta}}{\omega_n^{\frac{1}{n}}}\biggr)^n.
\end{equation*}
\end{proof}
\begin{oss}\label{rem3}
We observe that 
\begin{equation}
\label{claim4.5}
\abs{u^{\sharp-1}(t_{\varepsilon,\beta}, +\infty)}\le\varepsilon^{n'\beta},
\end{equation}
indeed
    $$\mu(t_{\varepsilon,\beta})=\lim_{t\rightarrow t_{\varepsilon,\beta}^+}\mu(t)\leq \varepsilon^{n'\beta},$$
    since $\mu$ is right continuous.
\end{oss}

\begin{prop} \label{prop2} Let $u$ be the solution to \eqref{mainproblem}, let $u^\sharp$ be its Schwartz rearrangement,and let $I$ and $t_{\varepsilon,\beta}$ defined as in \eqref{I} and \eqref{t_epsilon}, respectevelly. 
Then, there exists a positive constant $K_3=K_3(n)$ and a real number $\theta_4=\theta_4(n)$ such that 
\begin{equation}\label{claim5}
    \abs{u^{\sharp-1}(I \cap (0,t_{\varepsilon,\beta}))}<K_3\varepsilon^{\theta_4},
\end{equation}
where $K_3=5n^n\omega_n$, for every $\alpha,\beta\in \mathbb{R}$ such that  $0<\beta<\alpha<1$, $\beta<\dfrac{1}{n'}$.
\end{prop}
\begin{proof}
From  \eqref{nu} and  \eqref{mu}, we have
  %  $$n^2\omega_n^{\frac{2}{n}}=\dfrac{-\nu'(t)}{\nu(t)^{2-\frac{2}{n}}}\int_0^{\nu(t)}f^*(r)\;dr$$
   % and from \eqref{mu}
   % $$n^2\omega_n^{\frac{2}{n}}\leq \dfrac{-\mu'(t)}{\mu(t)^{2-\frac{2}{n}}}\int_0^{\mu(t)}f^*(r)\;dr,$$

    $$\dfrac{-\nu'(t)}{\nu(t)^{2-\frac{2}{n}}}\int_0^{\nu(t)}f^*(r)\;dr\leq \dfrac{-\mu'(t)}{\mu(t)^{2-\frac{2}{n}}}\int_0^{\mu(t)}f^*(r)\;dr\le \dfrac{-\mu'(t)}{\mu(t)^{2-\frac{2}{n}}}\int_0^{\nu(t)}f^*(r)\;dr $$
    and, consequently, using \eqref{muquasinu}, for all  $t\in (0,t_{\varepsilon,\beta})$, it holds
    \begin{equation}
        \label{deriv}
        -\mu'(t)\geq -\nu'(t)\left(\dfrac{\mu(t)}{\nu(t)}\right)^{2-\frac{2}{n}}\geq -\nu'(t)\biggl(1-\dfrac{n^2\omega_n^{\frac{2}{n}}\varepsilon}{\nu(t)}\biggr)^{2-\frac{2}{n}}\geq -\nu'(t)\biggl(1-n^2\omega_n^{\frac{2}{n}}\varepsilon^{1-n'\beta}\biggr)^{2-\frac{2}{n}}.
    \end{equation}
 Moreover, it is easily seen that 
    $$\abs{u^{\sharp-1}(I^c \cap (0,t_{\varepsilon,\beta}))}=\abs{u^{\sharp-1}(I^c\cap (0,t_{\varepsilon,\beta}))\cap \{|\nabla u^\sharp|=0\}}+\int_{I^c\cap (0,t_{\varepsilon,\beta})}\int_{u^\sharp=t}\dfrac{1}{\abs{\nabla u^\sharp}}\geq \int_{I^c\cap (0,t_{\varepsilon,\beta})}-\mu'.$$
    Hence, from \eqref{deriv} and the  fact that the distribution function of $v$ is absolutely continuous, we get

    \begin{align*}
        \abs{u^{\sharp-1}(I^c\cap (0,t_{\varepsilon,\beta}))} &\geq \biggl(1-n^2\omega_n^{\frac{2}{n}}\varepsilon^{1-n'\beta}\biggr)^{2-\frac{2}{n}}\int_{I^c\cap (0,t_{\varepsilon,\beta})}-\nu'\\&=\biggl(1-n^2\omega_n^{\frac{2}{n}}\varepsilon^{1-n'\beta}\biggr)^{2-\frac{2}{n}}\abs{v^{-1}(I^c\cap (0,t_{\varepsilon,\beta}))}\\\vspace{1mm}
        &\geq \abs{v^{-1}(I^c\cap (0,t_{\varepsilon,\beta}))} \left(1-\left(2-\frac{2}{n}\right)n^2\omega_n^{\frac{2}{n}}\varepsilon^{1-n'\beta} \right),
    \end{align*}
where the last inequality is the Bernoulli inequality, which holds true as 
\begin{equation}\label{def:vareps:0}
    \varepsilon<\varepsilon_0:=\left(\frac{1}{n^2\omega_n^\frac{2}{n}}\right)^\frac{1}{1-n'\beta}.
\end{equation}
    So, we have that 
    \begin{equation}
    \label{usharp}
    \begin{aligned}
          \abs{u^{\sharp-1}(I\cap (0,t_{\varepsilon,\beta}))}&=1-\abs{u^{\sharp-1}(I^c\cap (0,t_{\varepsilon,\beta}))}-\abs{u^{\sharp-1}(t_{\varepsilon,\beta},+\infty)}\\&
          \leq 1- \abs{u^{\sharp-1}(I^c\cap (0,t_{\varepsilon,\beta}))}- \mu(t_{\varepsilon,\beta})\\
          &\leq 1- \abs{v^{-1}(I^c\cap (0,t_{\varepsilon,\beta}))}\left(1-\left(2-\frac{2}{n}\right)n^2\omega_n^{\frac{2}{n}}\varepsilon^{1-n'\beta} \right) - \nu(t_{\varepsilon, \beta}) + n^2\omega_n^{\frac{2}{n}} \varepsilon \\&
          = \abs{v^{-1}(I\cap (0,t_{\varepsilon,\beta}))}+ n^2\omega_n^{\frac{2}{n}}\left(\left(2-\frac{2}{n}\right)\varepsilon^{1-n'\beta}+ \varepsilon\right).
          %\abs{v^{-1}(I)}+c_0\varepsilon^{1-n'\beta}=.
    \end{aligned}
    \end{equation}
  
Arguing as in \cite[Lemma 3.2]{polyaquantitativa}, we can deduce that

\begin{align*}
    \abs{v^{-1}(I)}&=\int_I\int_{v=t}\dfrac{1}{\abs{\nabla v}}\\&=
    \int_I\int_{\{v=t\}\cap \{\abs{\nabla v}\leq \delta\}}\dfrac{1}{\abs{\nabla v}}+\int_{ \{v=t\}\cap \{\nabla v > \delta\}}\dfrac{1}{\abs{\nabla v}}\\&\leq \abs{\{\abs{\nabla v}\leq\delta\}}+\dfrac{1}{\delta}\int_I\int_{ \{v=t\}\cap \{\abs{\nabla v}>\delta\}}\dfrac{1}{\abs{\nabla v}}\\&\leq \abs{\{\abs{\nabla v}\leq\delta\}}+\dfrac{1}{\delta}\int_I n\omega_n \left(\dfrac{\nu(t)}{\omega_n}\right)^{\frac{n-1}{n}}\\&=\abs{\{\abs{\nabla v}\leq\delta\}}+\dfrac{1}{\delta}n\omega_n^{\frac{1}{n}}\abs{I}
\end{align*}
and, by Proposition \ref{delta0} and Lemma \ref{lemmaI}, 
\begin{equation}
\label{vmenouno}
    \abs{v^{-1}(I\cap(0,t_{\varepsilon,\beta}))}\leq \abs{v^{-1}(I)}\leq \omega_n n^n\delta^n+\dfrac{n\omega_n^{\frac{1}{n}}}{\delta}\varepsilon^{1-\alpha}.
\end{equation}
Combining \eqref{usharp} and \eqref{vmenouno}, we obtain, setting $\delta=\varepsilon^q$, for $0<q<1-\alpha$
\begin{equation*}
    \abs{u^{\sharp-1}(I\cap (0,t_{\varepsilon,\beta}))}\leq n^2\omega_n^{\frac{2}{n}}\left(\left(2-\frac{2}{n}\right)\varepsilon^{1-n'\beta}+ \varepsilon\right)+\omega_n n^n\varepsilon^{nq}+n\omega_n^{\frac{1}{n}}\varepsilon^{1-\alpha-q}
                .
\end{equation*}
Choosing $$\theta_4=\min\{1-n'\beta, 1-\alpha-q, nq\},$$ 
 we get the desired claim for a positive constant $K_3$.
\end{proof}

\begin{prop}\label{Musharp}
Let $u$ be the solution to \eqref{mainproblem} and let $M_{u^\sharp}$ be the quantity defined in \eqref{eumu}. Then, there exist $\theta_5=\theta_5(n)>0$ and $K_4=K_4(n)$ such that
    
   $$ M_{u^\sharp}(\varepsilon^r)\leq K_4\varepsilon^{\theta_5},$$
   where $K_4=7n^n\omega_n$.
\end{prop}
\begin{proof}
We have 
\begin{align*}
    M_{u^\sharp}(\delta) &\leq \abs{\{|\nabla u^\sharp| < \delta\}}\\
    &=\abs{\left\{|\nabla u^\sharp|<\delta\right\}\cap u^{\sharp-1}(I^c\cap (0,t_{\varepsilon,\beta}))}\\
    &+\abs{\left\{|\nabla u^\sharp|<\delta\right\}\cap u^{\sharp-1}(I\cap (0,t_{\varepsilon,\beta}))}\\
    &+\abs{\left\{|\nabla u^\sharp|<\delta\right\}\cap u^{\sharp-1}(t_{\varepsilon,\beta}, +\infty)}\\
    &\leq \abs{\left\{|\nabla u^\sharp|<\delta\right\}\cap u^{\sharp-1}(I^c\cap (0,t_{\varepsilon,\beta})}\\
    &+\abs{u^{\sharp-1}(I\cap (0,t_{\varepsilon,\beta}))}+\mu(t_{\varepsilon,\beta})\\
    &\le \dfrac{1}{n^n}\biggl(\delta+\dfrac{\varepsilon^{\alpha-\beta}}{n\omega_n^{\frac{1}{n}}}\biggr)^n+ K_3\varepsilon^{\theta_4}+\varepsilon^{n'\beta},  
    %<\dfrac{\omega_n }{K_3}\biggl(\varepsilon^q \abs{\Omega}^{\frac{1}{n}}+\dfrac{\varepsilon^{\alpha-\beta} \abs{\Omega}^{\frac{1}{n}}}{n\omega_n^{\frac{1}{n}}}\biggr)^n+K_4\varepsilon^{\theta}+\abs{\Omega}\varepsilon^{n'\beta}
    %\biggl(\delta+\varepsilon^{\alpha-\beta}\biggr)^2+K_1\varepsilon^{\theta}+\varepsilon^{2\beta},
  \end{align*}
  where in the last inequality we have used \eqref{claim4}, \eqref{claim4.5} and  \eqref{claim5}.
  So, evaluating $M_{u^\sharp}(\delta)$ in $\varepsilon^r$, we obtain
  $$M_{u^\sharp}(\varepsilon^r)\leq \dfrac{1}{n^n}\biggl(\varepsilon^r+\dfrac{\varepsilon^{\alpha-\beta}}{n\omega_n^{\frac{1}{n}}}\biggr)^n+ K_3\varepsilon^{\theta_4}+\varepsilon^{n'\beta},$$
 so we can conclude, by setting

 $$\theta_5=\min\{ nr, n(\alpha-\beta), \theta_4, n'\beta  \},$$
obtaining the desired claim for a positive constant $K_4$.
\end{proof}
We are now in position to conclude with the proof of Theorem \ref{confrontol1}. 

\begin{proof}[Proof of Theorem \ref{confrontol1}] Firstly, let us assume $\abs{\Omega}=1$, and $\norma{f}_{1}=1$ and $\varepsilon < \varepsilon_0$. 
If this is not the case, i.e. $\varepsilon>\varepsilon_0$, the thesis \eqref{esti_L1} by choosing $ \tilde C_2$ sufficiently big.

Other,  from Theorem \ref{polya_quant}, we have that there exist two positive exponents $r$ and $s$ such that
 \begin{equation}
    \inf_{x_0 \in \R^n} \dfrac{ \int_{\R^n} \abs{u(x)-u^\sharp(x+x_0)} \, dx }{\norma{\nabla u^\sharp}_2}\leq C(n) \left[M_{u^\sharp}(E(u)^r)+E(u)\right]^s,
 \end{equation}
 where 
$$
E(u)=\dfrac{\int_{\Omega}\abs{\nabla u}^2-\int_{\Omega^{\sharp}}|\nabla u^{\sharp}|^2}{\int_{\Omega^{\sharp}}|\nabla u^{\sharp}|^2}.
$$
By using Lemma \ref{lem_grad}, we have that
 \begin{equation}\label{energy}
     E(u)\leq 2\dfrac{||v-u^{\sharp}||_{\infty}}{\int_{\Omega^{\sharp}}|\nabla u^{\sharp}|^2}.
 \end{equation}
 Hence, using \eqref{energy} and Proposition \ref{Musharp}, we get 
 \begin{equation*}
    \inf_{x_0 \in \R^n} \dfrac{||u-u^\sharp(\cdot + x_0)||_{L^1(\R^n)}}{\left(\int_{\Omega^{\sharp}}\abs{\nabla u^{\sharp}}^2\right)^{\frac{1}{2}}}\leq C(n)\left[K_4\left(\dfrac{2||v-u^{\sharp}||_{\infty}}{\int_{\Omega^\sharp}\abs{\nabla u^{\sharp}}^2}\right)^{\theta_5}+\dfrac{2||v-u^{\sharp}||_{\infty}}{\int_{\Omega^\sharp}|\nabla u^{\sharp}|^2}\right]^s.
 \end{equation*}
 So, there exists $K_5=K_5(n)$, such that 
 \begin{equation}
     \inf_{x_0 \in \R^n} \dfrac{||u-u^\sharp(\cdot + x_0)||_{L^1(\R^n)}}{\left(\int_{\Omega^{\sharp}}\abs{\nabla u^{\sharp}}^2\right)^{\frac{1}{2}}}\leq 
     K_5\left(\dfrac{||v-u^{\sharp}||_{\infty}}{\int_{\Omega^{\sharp}}\abs{\nabla u^\sharp}^2}\right)^{\tilde{\theta}_1},
 \end{equation}
 being $\tilde{\theta}_1=s\theta_5 $. 
 Moreover, if we apply the Talenti'result \eqref{stima:lq:gradv} with $q=2$ there exists $K_1=K_1(n,2)$ such that
$$||\nabla u^\sharp||_2\leq ||\nabla v||_2\leq K_1\norma{f}_{\frac{2n}{n+2}}.$$
 Hence, it follows that there exists a constant $\widetilde C_2>0$ such that
 \begin{equation}
     {\inf_{x_0 \in \R^n}||u-u^\sharp(\cdot + x_0)||}_{L^1(\R^n)}\leq 
    \widetilde C_2{||v-u^{\sharp}||}^{\tilde{\theta}_1}_{\infty} \norma{f}_{\frac{2n}{n+2}}^{1-2\tilde{\theta}_1}.
 \end{equation}

 Finally, we remove the additional assumptions $\abs{\Omega}=1$ and $||f||_1=1$ using the rescaling properties in Section \ref{riscalamento}.

\end{proof}
\begin{oss}
    We can improve estimate \eqref{esti_L1} by replacing the $L^1$-norm with the $L^q$- norm, with $q \in \left[1,\frac{2n}{n-2}\right[$, by exploiting the classical Gagliardo-Nirenberg estimate (see \cite[pag.125]{nirenberg}). In particular, for $q=2$, we obtain
\begin{equation}\label{gagliardo_l2}
\inf_{x_0 \in \R^n}||u-u^\sharp(\cdot + x_0)||_{L^2(\R^n)}^{\frac{n+2}{2}}\leq K_6\dfrac{\norma{f}_{\frac{2n}{n+2}}^{2+n-4\tilde{\theta}_1}}{\norma{f}_{1}^{1-3\tilde{\theta}_1}}\abs{\Omega}^{1+\left(\frac{2}{n}-1\right)\tilde{\theta}_1}||v-u^{\sharp}||_{\infty}^{\tilde{\theta}_1}.
\end{equation}
  where $K_6=K_6(n)>0$ is a positive constant.

\end{oss}
\section{Proof of Theorem \ref{confrontof}}
\label{section5}
The aim of this Section is to prove Theorem \ref{confrontof}. We are assuming that $f$ is in  $L^2(\R^n)$,  $\abs{\Omega}=1$ and $||f||_1=1$.

Furthermore, without loss of generality, we can assume that $x_0=0$ achieves the infimum in 
\begin{equation}
\label{infpolya}
    \inf_{x_0\in \R^n} \norma{u-u^{\sharp}(\cdot+x_0)}_{L^2(\R^n)}=\norma{u-u^{\sharp}}_{L^2(\R^n)},
\end{equation}
 and by the arguments in \cite{polyaquantitativa} it is possible to show that 
 \begin{equation}
     \label{quasialpha}
     \abs{\Omega \Delta \Omega^\sharp} \leq \tilde C_4\varepsilon^\frac{1}{4}.
 \end{equation}
Indeed, following \cite[Lemma 4.5]{polyaquantitativa},  $\Omega^\sharp$ is concentric to the ball which achieves the infimum in $$
\alpha(U_{t_0})= \inf_{x_0 \in \R^n} \left\{\frac{\abs{U_{t_0} \Delta B_{t_0}}}{\abs{B_{t_0}}}\,|\, \abs{U_{t_0}}=\abs{B_{t_0}}\right\},$$ 
with $t_0 \leq C \varepsilon^\frac{1}{4}$.

Moreover,  using \eqref{distr}, we have  
\begin{equation}
\label{kjbv}
    \mu(t_0) \geq \nu(t_0+\varepsilon)= 1 - \int_0^{t_0 + \varepsilon} \dfrac{n^2\omega_n^{\frac{2}{n}}}{\nu^{-1+\frac{2}{n}}\dashint_0^{\nu(r)}f^*}
    \geq1-\int_t^{t+\varepsilon} \dfrac{n^2\omega_n^{\frac{2}{n}}}{\abs{\Omega}^{-1+\frac{2}{n}}}\dfrac{\abs{\Omega}}{\norma{f}_1}=1- n^2\omega_n^{\frac{2}{n}}(t_0+\varepsilon).
\end{equation}
Eventually, taking into account Theorem \ref{sopralivelli}, formula \eqref{kjbv} implies
\begin{align*}
    \abs{\Omega \Delta \Omega^\sharp}  &= 2 \abs{\Omega \setminus \Omega^\sharp }\\
    &=2 \abs{\Big(\{u >t_0\} \cup\{u \leq t_0\}\Big) \setminus \Omega^\sharp }\\
     &=2 \abs{\{u >t_0\} \setminus \Omega^\sharp }+  2 \abs{\{u \leq t_0\} \setminus \Omega^\sharp }\\
     &\leq2 \abs{\{u >t_0\} \setminus B_{t_0} }+  2 \abs{\{u \leq t_0\} }\\
     &= \alpha(U_{t_0}) + (1- \mu(t_0)) \\
     &\leq \frac{\tilde C_1 }{\mu(t_0)} \varepsilon^{\frac{1}{3}} + n^2\omega_n^{\frac{2}{n}}(t_0+\varepsilon) \\
    & \leq 2\tilde C_1 \varepsilon^{\frac{1}{3}} + n^2\omega_n^{\frac{2}{n}}(C\varepsilon^{\frac{1}{4}}+\varepsilon)\\ 
    &\leq \tilde C_4 \varepsilon^{\frac{1}{4}}
\end{align*}

The main idea of this section is to apply the quantitative Hardy-Littlewood inequality (Theorem \ref{cianchi_hardy}) to $h=v$, the solution to \eqref{symmetricproblem}, 
 and the function $g$ defined as follows
\begin{equation}\label{g}
    g=f\chi_{\Omega^\sharp\cap \Omega}. 
\end{equation} 
Next lemma ensures us that the functions $h$ and $g$ are admissible functions for Theorem \ref{cianchi_hardy}.
\begin{comment}
    \newpage
The aim of this Section is to prove Theorem \ref{confrontof}, hence from now on $f$ is in  $L^2(\R^n)$. 

Without loss of generality, we can assume, up to translations, that 
$$\inf_{x_0\in \R^n} \norma{u-u^{\sharp}(\cdot+x_0)}_{L^2(\R^n)}=\norma{u-u^{\sharp}}_{L^2(\R^n)},$$
and

As a first step we want to apply the quantitative Hardy-Littlewood inequality, Theorem \ref{cianchi_hardy}, choosing $h=v$, where $v$ is the solution to \eqref{symmetricproblem}, 
 and choosing the function $g$ as follows
\begin{equation}\label{g}
    g=f\rvert_{\Omega^\sharp\cap \Omega}, 
\end{equation}
and extended equal to $0$ in $\Omega^\sharp\setminus\Omega$.
\end{comment}
 
%We recall the following notations introduced in \cite{cianchi_ferone}. Let $p\in (1,+\infty)$ and $q\in [1,+\infty)$, 
%\begin{equation}\label{lorentz_cianchi}
%\norma{g}_{\Lambda^q_p(\Omega)}=\left(\int_0^{\abs{\Omega}}g^\ast(s)^q \theta'_q(s)\;ds\right)^{1/q},
%\end{equation}
%where, for $s\in [0, \abs{\Omega})$,
%\begin{equation}\label{theta}
 %   \theta_p(s)=\left(\int_0^s \left(-(v^\ast) '  (\sigma)\right)^{-1/(p-1)}  d\sigma \right)^{1/p'}
%\end{equation}
%and we set 
\begin{lemma}
\label{gigino}
    Let $f\in L^2(\Omega)$ and let $g$ be the function  defined in \eqref{g}.  Then , there exist $p\in (1,+\infty)$ and $q\in [1,+\infty)$ such that, we have

    \begin{equation}\label{cian}
    \inf_{x_0 \in \R^n}\norma{g-g^\sharp(\cdot+x_0)}_{L^m(\R^n)}\leq \left(\dfrac{1}{2^{p+1}eq}\norma{g}_{\Lambda_p^q(\Omega)}^{pq}\int_{\Omega^\sharp}(g^\sharp-g)v\right)^{\frac{1}{1+pq}},
\end{equation}
where
\begin{equation*}
    m=\dfrac{qp+1}{p+1},
\end{equation*}
 and moreover, it holds $m<2$.
\end{lemma}
\begin{proof}
    In order to prove \eqref{cian}, we apply Theorem \ref{cianchi_hardy} to the function $h=v$ and $g$ defined in \eqref{g}. Let us check that the assumptions of the theorem are satisfied, namely 
\begin{equation*}
    \theta_p(s)<\infty;\quad\norma{g}_{\Lambda^q_p(\Omega)}<\infty, 
\end{equation*}
for some $p,q\in [1,+\infty)$. Indeed,
\begin{align*}
    \theta_p(s)&=\left(\int_0^s\left(\dfrac{1}{n^2\omega_n^{\frac{2}{n}}t^{1-\frac{2}{n}}}\dashint_0^tf^*(r)\;dr\right)^{-\frac{1}{p-1}}dt\right)^{\frac{1}{p'}}\\&
    \leq \left(\dfrac{n^2\omega_n^{\frac{2}{n}}\abs{\Omega}}{\norma{f}_1}\right)^{\frac{1}{p}}\left(\int_0^st^{(1-\frac{2}{n})\frac{1}{p-1}}\right)^{\frac{1}{p'}}<+\infty, 
\end{align*}
for all $p>1$, and
    \begin{align*}
\norma{g}_{\Lambda_p^q}^q=\int_0^{\abs{\Omega}}(g^*)^q\theta_p'(s)&=\frac{1}{p'}\int_0^{\abs{\Omega}}(g^*)^q\biggl(\int_0^s\biggl(\dfrac{1}{n^2\omega_n^{\frac{2}{n}}t^{1-\frac{2}{n}} }\dashint_0^tf^*\biggr)^{-\frac{1}{p-1}}\biggr)^{-\frac{1}{p}}\biggl(\dfrac{1}{n^2\omega_n^{\frac{2}{n}}s^{1-\frac{2}{n}}}\dashint_0^sf^*\biggr)^{-\frac{1}{p-1}}\, ds\\&\leq
   \dfrac{1}{p'}\displaystyle\int_0^{\abs{\Omega}}(f^*)^q\left[\dfrac{\biggl(\dfrac{1}{n^2\omega_n^{\frac{2}{n}}s^{2-\frac{2}{n}}}\displaystyle\int_0^sf^*\biggr)^{-\frac{1}{p-1}}}{\displaystyle\int_0^s\biggl(\dfrac{1}{n^2\omega_n^{\frac{2}{n}}t^{2-\frac{2}{n}}}\displaystyle\int_0^tf^*\biggr)^{-\frac{1}{p-1}}}\right]^{\frac{1}{p}}\biggl(\dfrac{1}{n^2\omega_n^{\frac{2}{n}}s^{2-\frac{2}{n}}}\int_0^sf^*\biggr)^{-\frac{1}{p}} \, ds.    
   \end{align*}
  %\leq \dfrac{1}{p'}\int_0^{\abs{\Omega}}(f^*)^q\biggl(\int_0^s\biggl(\dfrac{1}{n^2\omega_n^{\frac{2}{n}}t^{1-\frac{2}{n}}}\int_0^sf^*\biggr)^{-\frac{1}{p-1}}\biggr)^{-\frac{1}{p}}\left[\biggl(\dfrac{1}{n^2\omega_n^{\frac{2}{n}}t^{1-\frac{2}{n}} }\int_0^sf^*\biggr)^{-\frac{1}{p-1}}\right]^{\frac{1}{p}+\frac{p-1}{p}}
    We observe that, if   $t\le s$, we have
    \begin{itemize}
        \item          $\displaystyle{\int_0^tf^*\leq \int_0^s f^*}, $
        since $f^\ast$ is positive;
        \item  $f^*(t)\geq f^*(s)$, since $f^\ast$ is decreasing.
    \end{itemize}
     It follows that
\begin{align*}
    \norma{g}_{\Lambda_p^q}^q&\leq \dfrac{1}{p'}\int_0^{\abs{\Omega}}(f^*)^{q}\left(\dfrac{s^{(2-\frac{2}{n})\frac{1}{p-1}}}{\displaystyle\int_0^s t^{(2-\frac{2}{n})\frac{1}{p-1}}}\right)^{\frac{1}{p}}\biggl(\dfrac{1}{n^2\omega_n^{\frac{2}{n}}s^{2-\frac{2}{n}}}\int_0^sf^*\biggr)^{-\frac{1}{p}}\\&
    \leq K_7(p,n)\int_0^{\abs{\Omega}}(f^*)^{q-\frac{1}{p}}\left(\dfrac{s^{(2-\frac{2}{n})\frac{1}{p-1}}}{ s^{(2-\frac{2}{n})\frac{1}{p-1}+1}}\right)^{\frac{1}{p}}
    s^{\left(1-\frac 2 n \right) \frac 1 p }
    \\
    &
    \leq K_7(p,n) \int_0^{\abs{\Omega}} (f^*)^{q-\frac{1}{p}} s^{-\frac{2}{np}}\;ds =K_7(p, n)\int_0^{\abs{\Omega}}\left(f^*s^{\frac{p-\frac 2 n}{qp-1}}\right)^{q-\frac{1}{p}}\;\dfrac{ds}{s},
\end{align*}
where $$K_7(p, n)=\frac{1}{p'}\left[\left(2-\frac{2}{n}\right)\frac{1}{p-1}+1\right]^{\frac{1}{p}}(n^2\omega_n^{\frac{2}{n}})^{\frac{1}{p}}.$$
So, we have that 
\begin{equation*}
    \norma{g}^q_{{\Lambda}_p^q}\leq K_7(p,n)\norma{f}^l_{L^{k,l}(\R^n)},
\end{equation*}
where 
$$k=\dfrac{qp-1}{p-\frac 2 n
}, \qquad l=\dfrac{qp-1}{p}.$$
By imposing $l>
0$ and $0< k<2$, we obtain the following constraints on $p$ and $q$:
\begin{equation}\label{vincoli}
    p>1, \qquad \dfrac{1}{p}\leq q < 2+\left(1-\frac{4}{n}\right)\dfrac{1}{p}.
\end{equation}
So, by the classical  embedding Theorem for Lorentz space (see \cite{Lorenz_emb}), we have 
$$L^{2,2}\hookrightarrow L^{k,l}.$$
Now,  we are in position to use Theorem \ref{cianchi_hardy} and we obtain
\begin{equation*}
   \inf_{x_0 \in \R^n} \norma{g-g^\sharp(\cdot+x_0)}_{L^m(\R^n)}\leq \left(\dfrac{1}{2^{p+1}\,e\, q}\norma{g}_{\Lambda_p^q(\Omega)}^{pq}\int_{\Omega^\sharp}(g^\sharp-g)v\right)^{\frac{1}{1+pq}},
\end{equation*}
where
\begin{equation*}
    m=\dfrac{qp+1}{p+1}=\dfrac{2n(k-l)+2kl}{2k+n(k-l)}.
\end{equation*}
Moreover, we observe that, from \eqref{vincoli}, it follows that  $m<2$.

\end{proof}
Let us prove now Theorem \ref{confrontof}.
\begin{proof}[Proof of Theorem \ref{confrontof}]
Firstly, let us observe that 
\begin{align*}
 \inf_{x_0\in \R^n}||f-f^\sharp(\cdot+x_0)||_{L^m(\mathbb{R}^n)}  &\leq ||f-f^\sharp||_{L^m(\mathbb{R}^n)}  \\ &\leq   \norma{f-g}_{L^m(\mathbb{R}^n)}+||g-g^\sharp||_{L^m(\mathbb{R}^n)}+||f^\sharp-g^\sharp||_{L^m(\mathbb{R}^n)}.
\end{align*}
Let us estimate from above the quantities $||f^\sharp-g^\sharp||_{L^m(\mathbb{R}^n)}$ and $\norma{f-g}_{L^m(\mathbb{R}^n)}$:
\begin{equation}
\label{f_sharp-g_sharp}
\begin{aligned}
    ||f^\sharp-g^\sharp||_{L^m(\mathbb{R}^n)}\leq \norma{f-g}_{L^m(\mathbb{R}^n)}&=\norma{f}_{L^m(\Omega\setminus\Omega^\sharp)}\\
    & \leq  \abs{\Omega\setminus \Omega^\sharp}^{\frac{1}{m}-\frac {1} {2}} \norma{f}_{2} \\
    &\leq  \tilde C_4 \norma{f}_{2}\varepsilon^{\frac{2-m}{8m}}
\end{aligned}
\end{equation}
where in the last inequality we have used Holder's inequality and \eqref{quasialpha}.

It only remains to estimate $||g-g^{\sharp}||_{L^m(\R^n)}$, and thanks to Lemma \ref{gigino} and \eqref{gagliardo_l2}, it is enough to estimate the following quantity
\begin{equation}
    \begin{aligned}
    \int_{\Omega^\sharp}(g^\sharp-g)v &\leq \int_{\Omega^\sharp}f^\sharp v-\int_{\Omega}fu+\int_{\R^n}f(u-u^\sharp)+\int_{\Omega^\sharp}f(u^\sharp-v)\\&
    = \int_{\Omega^\sharp} \abs{\nabla v}^2-\int_{\Omega}\abs{\nabla u}^2+\int_{\R^n}f(u-u^\sharp)
    \\& \leq2\varepsilon+\norma{f}_{2}||u-u^\sharp||_{L^2(\R^n)}\\& 
    \leq2\varepsilon+
K_6\norma{f}_{2}\left({\norma{f}_{{\frac{2n}{n+2}}}^{2+n-4\tilde{\theta}_1}}\varepsilon^{\tilde{\theta}_1}\right)^\frac{2}{n+2}\\
&\leq 2\varepsilon+K_6\norma{f}_2^{3-\frac{8}{n+2}\tilde{\theta}_1}\varepsilon^{\frac{2}{n+2}\tilde{\theta}_1}.
\end{aligned}
\end{equation}
Hence,
\begin{equation}\label{g-gsharp}
%\begin{multlined}
     ||g-g^\sharp||_{L^m(\R^n)}%\qquad \qquad \\  \qquad \qquad
\le \left(\dfrac{1}{2^{p+1}\,e\, q}\right)^{\frac{1}{1+pq}}
\norma{f}_2^{\frac{pq-1}{1+pq}}
    \left(2\varepsilon+K_6\norma{f}_2^{3-\frac{8}{n+2}\tilde{\theta}_1}\varepsilon^{\frac{2}{n+2}\tilde{\theta}_1}\right)^{\frac{1}{1+pq}}.
%\end{multlined}
\end{equation}

%Moreover, we recall that 
%\begin{equation}\label{già_labellata}
  %  \int_{\Omega^\sharp} f^\sharp v- \int_{\Omega} %f^\sharp u<\varepsilon. 
%\end{equation}
Combining   \eqref{f_sharp-g_sharp} and  \eqref{g-gsharp}, there exists a constant $\widetilde{C}_3$ depending only on $m$ and the dimension $n$, such that
$$
\inf_{x_0\in \R^n}\norma{f-f^{\sharp}(\cdot + x_0)}_{L^m(\R^n)} \leq 
\widetilde C_3\norma{f}_2^{\tilde{\theta}_2} \varepsilon^{\tilde{\theta}_3},
$$
where
$$\tilde{\theta}_2=\max\left\{1,\left(2+pq-\frac{8 \tilde{\theta}_1}{n+2}\right)\left(\frac{1}{1+pq}\right)\right\},\qquad\tilde{\theta}_3=\min\left\{\frac{2-m}{8m},\frac{1}{1+pq},\left(\frac{2 \tilde{\theta}_2}{n+2}\right)\frac{1}{1+pq}\right\}.$$
\end{proof}

 \begin{oss}
We explicitly observe that in Theorem \ref{confrontof} we cannot obtain a comparison result for the quantity $||f-f^\sharp||_{L^2(\R^n)}$. Indeed, let $B$ be the unitary ball, centered at the origin, in the plane, and let us consider the following problem and its symmetrized:

\begin{equation*}
\label{problem_counter}
    \begin{cases}
        -\Delta u_\sigma =f_\sigma&\text{in}\; B\\
        u_\sigma=0 &\text{on}\;\partial B.
    \end{cases}; \qquad\qquad  \begin{cases}
        -\Delta v_\sigma =f^\sharp_\sigma&\text{in}\,B\\
        v_\sigma=0 &\text{on}\,\partial B,
    \end{cases}
    \end{equation*}
where 
\begin{equation*}
f_\sigma(x)=1+\sigma^{-1}\chi_{B_\sigma(1/2,0)}(x).
\end{equation*}
and, consequently, 
\begin{equation*}
f^\sharp_\sigma(x)=1+\sigma^{-1}\chi_{B_\sigma(0,0)}(x) .
\end{equation*}
With these choices, we obtain
\begin{equation*}
    \norma{f_\sigma-f^\sharp_\sigma}_{L^2(B)}=\sqrt{2 \pi}
\end{equation*}
and 
\begin{equation*}
    \norma{f_\sigma-f^\sharp_\sigma}_{L^r(B)}=\left(2\pi\right)^{1/r} \sigma^{2/r-1} \to 0
\end{equation*}
as  $\sigma\to 0$ and $r <2$.
In order to conclude, we refer now to  \cite[Theorem 8.16]{trudinger}, 
that states that the following estimate holds
$$\norma{w}_\infty\le C\left( \sup_{\partial\Omega} w + \norma{g}_{\frac{q}{2}}\right)$$
where $w$ solves in $\Omega$
$$
    -\Delta w =g
$$
with  $g\in L^{\frac{q}{2}}$, $q>2$. So, in our case, we can choose $w=v_\sigma-u_\sigma$ and any $r>1$, obtaining
\begin{equation*}\label{trudinger}
    \norma{v_\sigma-u_\sigma}_{L^\infty(B)}\leq C \norma{f^\sharp_\sigma-f_\sigma}_{L^r(B)},
\end{equation*}
so it is clear that, if $1\le r<2$, $\norma{v_\sigma-u_\sigma}_{L^\infty(B)}\to 0$.
\end{oss}

\section{Conclusions and Open Problems}

\begin{proof}[Proof of Theorem \ref{main_theorem}] 

The proof of the main Theorem follows directly by combining Theorem \ref{asimmetryindex}, Theorem \ref{confrontol1}, and Theorem \ref{confrontof}, by choosing  
$$\theta_1= \frac{1}{\tilde{\theta}_1}, \quad  \theta_2= \frac{1}{\tilde{\theta}_2}$$
and 
$$ \begin{aligned}
    &C_1=\frac{1}{\widetilde{C_1}} \abs{\Omega}^{\frac{2-n}{n}} \norma{f}_1,\quad
    C_2=\left(\frac{1}{\widetilde{C_2}}\dfrac{\norma{f}_{1}^{1-3\tilde{\theta}_1}}{\norma{f}_{\frac{2n}{n+2}}^{1-2\tilde{\theta}_1}}\abs{\Omega}^{-1-\left(\frac{2}{n}-1\right)\tilde{\theta}_1}\right)^\frac{1}{\tilde{\theta}_1}, \\
    &C_3=\left(\displaystyle{\widetilde{C_3} \norma{f}^{\tilde{\theta}_2}_2 \norma{f}_1^{1-\tilde{\theta}_3-\tilde{\theta}_2}\abs{\Omega}^{\frac{\tilde{\theta}_2}{2}+\frac{1-m}{m}+\left(\frac{n-2}{n}\right)\tilde{\theta}_3}}\right)^{-\frac{1}{\tilde{\theta}_3}},
\end{aligned}$$
where $\widetilde C_1$, $\widetilde C_2$ and $\widetilde C_3$ are the constants appearing, respectively, in Theorem \ref{asimmetryindex}, Theorem \ref{confrontol1}, and Theorem \ref{confrontof}, $\tilde {\theta}_1$ is the exponent in Theorem \ref{confrontol1} and $\tilde{\theta}_2$ and $\tilde{\theta}_3$ are the exponents appearing in Theorem \ref{confrontof}.
\end{proof}

We conclude with a list of open problems. 

\subsection*{Open problems}
\begin{itemize}
        \item It remains open the issue regarding the sharpness of the exponents of the quantities $\alpha(\Omega)$, $||u-u^\sharp||_{L^1(\mathbb{R}^n)} $ and   $||f-f^\sharp||_{L^1(\mathbb{R}^n)} $in Theorem \ref{main_theorem}.  
        As far as the exponent of $\alpha(\Omega)$, it may depend on the function $f$ and it remains open to establish if there is a function $f$ for which it is sharp.

\item It could be interesting to prove the quantitative result of the Talenti comparison result in the non-linear setting. We recall that the comparison result for the $p-$Laplace operator is proved in \cite{T2} and the relative  rigidity result can be found in \cite{masiello2023rigidity}
      
        \item In \cite{ANT} it is proved that a Talenti comparison result holds, when replacing the Dirichlet boundary conditions with the Robin boundary conditions with a positive boundary parameter. It could be interesting to study a quantitative version of this result, as it has already been proved in \cite{masiello2023rigidity} that the estimates are rigid.
    \end{itemize}

\label{section6}

%Moreover, we recall that 
%\begin{equation}\label{già_labellata}
  %  \int_{\Omega^\sharp} f^\sharp v- \int_{\Omega} %f^\sharp u<\varepsilon. 
%\end{equation}
%\addcontentsline{toc}{chapter}{Bibliografia}
\section{Appendix: $L^2-$ bound of the asymmetry}
\label{appendix}
    In Theorem \ref{asimmetryindex},  we  bound the Frankel asymmetry of the set $\Omega$ with  the quantity $||u^\sharp-v ||_{\infty}$. We observe that it is possible to obtain an analogous result in terms of $||v||_{2}-||u||_2$ and this kind of estimate might seem more natural, as the solution to \eqref{mainproblem} and \eqref{symmetricproblem} are always in $L^2$, while they are in $L^\infty$ only if $f\in L^p$ with $p>n/2$. 

Arguing as in the proof of Theorem \eqref{asimmetryindex}, the following estimate holds

     \begin{equation}
         \label{asiml2}
         \alpha^4(\Omega)\le C (\norma{v}_2^2-\norma{u}_2^2).
     \end{equation}
   Indeed, let us observe that, if we multiply \eqref{peromega}, i.e. 
   \begin{equation*}
        \begin{aligned}
       n^2\omega_n^{\frac{2}{n}}\mu^{2-\frac{2}{n}}(t)\left(1+\dfrac{2}{\gamma_n}\alpha^2(U_t)\right)&\le
      \left( \int_0^{\mu(t)}f^*\right) (-\mu'(t)),
        \end{aligned}
    \end{equation*}
   by the quantity $$ \dfrac{t\mu(t)^{1-2+\frac{2}{n}}}{n\omega_n^{2/n}}$$ and integrate between $0$ and $||u||_{\infty}$, we obtain
        \begin{equation}\label{moltiplicata}
            \int_0^{||u||_{\infty}} t\mu(t)+c_n \int_0^{||u||_{\infty}} t\mu(t)\alpha^2(U_t)\;dt \leq \int_0^{||u||_{\infty}} t\mu(t)^{\frac{2}{n}-1}\left(-\mu'(t)\right)\left(\int_0^{\mu(t)}f^*(s)ds\right) dt.
        \end{equation}
        We define now 
        $$F(l)=\int_0^lw^{\frac{2}{n}-1}\left(\int_0^w f^*(s)ds  \right)dw$$
         and, since $F(\cdot)$ is an increasing function and  the Talenti comparison in Theorem \ref{talenti:originale} holds,  we have 
       \begin{equation}
           F(\mu(t))<F(\nu(t)).
       \end{equation}
       So, \eqref{moltiplicata}, becomes 
       \begin{align*}
           \int_0^{||u||_{\infty}} t\mu(t)+\frac{2}{\gamma_n} \int_0^{||u||_{\infty}} t\mu(t)\alpha^2(U_t) dt &\leq -\int_0^{\norma{u}_{\infty}}-t\frac{d}{dt}(F(\mu(t)))\;dt\\&=\int_0^{\norma{u}_{\infty}}F(\mu(t))\leq\int_0^{\norma{v}_{\infty}}F(\nu(t))= \int_0^{\norma{v}_{\infty}}t\nu(t)\;dt.
        \end{align*}
            Then,
    \begin{equation*}
   \frac{2}{\gamma_n}\int_0^{s_\Omega}t\mu(t)\alpha^2(U_t)\leq \frac{2}{\gamma_n}\int_0^{\norma{u}_{\infty}}t\mu(t)\alpha^2(U_t)\leq {\norma{v}_2}^2-{\norma{u}_2}^2
    \end{equation*}
    and by definition of $s_\Omega$ in \eqref{essami},
    \begin{equation*}
      \frac{2}{\gamma_n}\int_0^{s_{\Omega}}t\mu(t)\alpha^2(U_t)\geq \abs{\Omega}\left(1-\dfrac{\alpha(U_t)}{4}\right)\dfrac{{s_{\Omega}}^2}{2}\dfrac{{\alpha^2({\Omega})}}{4}\geq \dfrac{\abs{\Omega}}{16}s_{\Omega}^2\alpha^2(\Omega).
    \end{equation*}
    Hence,
    \begin{equation}
        \label{differencel2somega}
   \dfrac{\abs{\Omega}}{16}s_{\Omega}^2\alpha^2(\Omega)\leq {\norma{v}_2}^2-{\norma{u}_2}^2.
    \end{equation}
    In order to conclude, we have to set $t_1$ such that 
    $$
   \nu(2t_1) = \abs{\Omega}\left(1- \frac{1}{8} \alpha(\Omega)\right),
    $$
    for which it holds the following bound
    \begin{equation*}
    \label{t_0geq}
        \begin{aligned}
        \frac{\abs{\Omega}}{8} \alpha(\Omega) &=\abs{\Omega}-\abs{\Omega}\left(1- \frac{1}{8} \alpha(\Omega)\right)= \nu(0)- \nu(2t_1)=- \int_{0}^{2t_1} \nu'(s)\,ds\\&= \int_{0}^{2t_1} \frac{n^2\omega_n^{\frac 2 n}}{\nu^{-1+\frac 2n}(t) \dashint_0^{\nu(t)}f^\ast}\leq \frac{n^2\omega_n^{\frac 2 n}}{\abs{\Omega}^{-1+\frac 2n} \norma{f}_1} 2t_1.
    \end{aligned}
    \end{equation*}
    Now let us distinguish 2 cases:
    \begin{itemize}
        \item $s_\Omega \geq t_1 \geq C \alpha(\Omega)$ then from \eqref{differencel2somega}, we obtain
        $$
         {\norma{v}_2}^2-{\norma{u}_2}^2 \geq  \dfrac{\abs{\Omega}}{16} C\alpha^4(\Omega);
        $$
        \item $s_\Omega <t_1$ we have
        \begin{equation}
        \label{93}
              {\norma{v}_2}^2-{\norma{u}_2}^2  = \int_0^{\norma{v}_\infty} 2t(\nu(t)-\mu(t))\, dt \geq \int_{t_1}^{2t_1} 2t(\nu(t)-\mu(t))\, dt.
        \end{equation}
        Since $s_{\Omega} < t_1 \leq t \leq 2t_1$, we have both
        \begin{equation}
        \label{nufinale}
        \nu(t) \geq \nu(2t_1)= \abs{\Omega}\left(1- \frac{1}{8} \alpha(\Omega)\right),
        \end{equation}
        and
        \begin{equation}\label{mufinale}
            \mu(t) \leq \abs{\Omega}\left(1- \frac{1}{4} \alpha(\Omega)\right).
        \end{equation}
      
        Consequently, combining \eqref{nufinale}, \eqref{mufinale} and \eqref{93} we have
        \begin{equation*}
        \begin{aligned}
            {\norma{v}_2}^2-{\norma{u}_2}^2 &\geq \int_{t_1}^{2t_1} 2t(\nu(t)-\mu(t))\, dt \geq \frac{\abs{\Omega}}{8} \alpha(\Omega)\int_{t_1}^{2t_1} 2t\, dt\\&\geq 
            \frac{\abs{\Omega}}{8} \alpha(\Omega)3t^2_1 \geq 3
            \frac{\abs{\Omega}}{8} \alpha^3(\Omega) 
            \geq 3
            \frac{\abs{\Omega}}{8} \alpha^4(\Omega). 
        \end{aligned}
        \end{equation*}
    \end{itemize}
    
\subsection*{Acknowledgements}
We would like to thank Francesco Della Pietra, Carlo Nitsch and Cristina Trombetti for the valuable advice that helped us to achieve these results.

The authors were partially supported by Gruppo Nazionale per l’Analisi Matematica, la Probabilità e le loro Applicazioni
(GNAMPA) of Istituto Nazionale di Alta Matematica (INdAM).   
\bibliographystyle{abbrv}
\bibliography{bibliografia}

\end{document}